\documentclass{amsart}
\usepackage{amsmath,amssymb,amsthm,mathtools}
\usepackage{tikz-cd}

\usepackage[utf8]{inputenc}
\usepackage[british]{babel}
\usepackage{lmodern}
\usepackage[T1]{fontenc}
\usepackage[babel=true]{microtype}
\usepackage[mathscr]{euscript}
\usepackage[linktoc=all]{hyperref}
\usepackage[capitalise]{cleveref}
\usepackage[backend=biber, bibencoding=utf8,style=alphabetic]{biblatex}
\usepackage{relsize}

\DefineBibliographyStrings{english}{%
	bibliography = {References},
}

\binoppenalty=\maxdimen
\relpenalty=\maxdimen

\theoremstyle{plain}
\newtheorem{thm}{Theorem}[subsection]
\newtheorem{lem}[thm]{Lemma}
\newtheorem{cor}[thm]{Corollary}
\newtheorem{prop}[thm]{Proposition}

\theoremstyle{definition}
\newtheorem{defn}[thm]{Definition}

\newtheorem{rem}[thm]{Remark}

\newtheorem{constr}[thm]{Construction}

\theoremstyle{plain} 
\newcommand{\thistheoremname}{}
\newtheorem{genericthm}[thm]{\thistheoremname}
\newenvironment{custom}[1]
  {\renewcommand{\thistheoremname}{#1}%
   \begin{genericthm}}
  {\end{genericthm}}
\theoremstyle{remark}

\DeclareMathOperator{\Hom}{Hom}
\DeclareMathOperator{\iHom}{\underline{Hom}}

\newcommand*{\from}{\colon}
\newcommand*{\defword}[1]{\emph{#1}}

\newcommand*{\defined}{\coloneqq}

\mathchardef\mhyphen="2D
\newcommand*{\xto}[1]{\xrightarrow{#1}}
\DeclareMathOperator*{\colim}{colim}
\DeclareMathOperator{\coker}{coker}
\DeclareMathOperator{\id}{id}
\DeclareMathOperator{\im}{im}

\DeclareMathOperator{\pr}{pr}

\DeclareMathOperator{\ev}{ev}

\makeatletter
\def\mydefbb#1{\expandafter\def\csname bb#1\endcsname{\mathbb{#1}}}
\def\mydefallbb#1{\ifx#1\mydefallbb\else\mydefbb#1\expandafter\mydefallbb\fi}
\mydefallbb ABCDEFGHIJKLMNOPQRSTUVWXYZ\mydefallbb
\makeatother
\newcommand*{\injto}{\hookrightarrow}

\newcommand*{\zl}{\mathbb{Z}_{\ell}}

\newcommand*{\solid}{\mathsmaller{\square}}
\newcommand*{\zld}{\mathbb{Z}_{\ell,\mathrm{disc}}}

\newcommand*{\cond}[1]{\underline{#1}}

\newcommand*{\Mod}{\mathrm{Mod}}

\DeclareMathOperator{\spa}{Spa}
\newcommand*{\geompt}{\spa(C,\mathcal{O}_C)}

\newcommand*{\lis}{\mathrm{lis}}

\newcommand*{\lotimes}{\otimes^{\bbL}}
\newcommand*{\rhom}{\mathrm{RHom}}

\newcommand*{\gm}{\mathbb{G}_m}

\DeclareMathOperator{\spec}{Spec}

\DeclareMathOperator{\res}{res}
\DeclareMathOperator{\fun}{Fun}

\DeclareMathOperator{\rep}{Rep}

\newcommand*{\bun}[1]{\mathrm{Bun}_{#1}}
\newcommand*{\lheck}{\overleftarrow{h}}
\newcommand*{\rheck}{\overrightarrow{h}}
\newcommand*{\heck}{\mathrm{Hck}}

\newcommand*{\Div}{\mathrm{Div}}

\DeclareMathOperator{\spd}{\mathrm{Spd}}

\newcommand*{\lsotimes}{\stackrel{\solid}{\otimes}\!\!{}^\mathbb{L}}

\newcommand*{\oo}{\mathcal{O}}
\newcommand*{\gr}{\mathrm{Gr}}
\newcommand*{\rshhom}{\mathrm{R}\mathscr{H}om}

\newcommand*{\et}{\mathrm{\acute{e}t}}

\newcommand*{\riHom}{\mathrm{R}\underline{\mathrm{Hom}}}

\newcommand*{\localheck}{\mathcal{H}\mathrm{ck}}

\DeclareMathOperator{\sat}{Sat}

\newcommand*{\cind}{c\text{-ind}}
\DeclareMathOperator{\perf}{Perf}
\DeclareMathOperator{\qcoh}{QCoh}
\DeclareMathOperator{\coh}{Coh}
\DeclareMathOperator{\nilp}{Nilp}

\newcommand{\op}{{\mathrm{op}}}

\DeclareMathOperator{\End}{End}

\DeclareMathOperator{\ind}{ind}
\DeclareMathOperator{\coind}{coind}
\DeclareMathOperator{\nm}{Nm}

\DeclareMathOperator{\maps}{Maps}
\DeclareMathOperator{\tr}{tr}

\DeclareMathOperator{\LMod}{LMod}

\DeclareMathOperator{\D}{\mathcal{D}}

\DeclareMathOperator{\LEnd}{LEnd}

\newcommand{\qc}{\mathrm{qc}}
\DeclareMathOperator{\exc}{Exc}
\newcommand{\bcenter}{\mathcal{Z}}

\usepackage{scalerel, stackengine}
\stackMath
\newcommand\widecheck[1]{%
\savestack{\tmpbox}{\stretchto{%
  \scaleto{%
    \scalerel*[\widthof{\ensuremath{#1}}]{\kern-.6pt\bigwedge\kern-.6pt}%
    {\rule[-\textheight/2]{1ex}{\textheight}}
  }{\textheight}%
}{0.5ex}}%
\stackon[1pt]{#1}{\scalebox{-1}{\tmpbox}}%
}
\DeclareRobustCommand{\gobblefive}[5]{}
\newcommand*{\SkipTocEntry}{\addtocontents{toc}{\gobblefive}}
\makeatletter
\def\subsection{\@startsection{subsection}{2}%
  \z@{.5\linespacing\@plus.7\linespacing}{.3\linespacing}%
  {\normalfont\bfseries}}
\makeatother
\usepackage{fullpage}
\pdfinterwordspaceon
\author{Konrad Zou}
\title{The categorical form of Fargues' conjecture for tori}
\begin{document}	
	\begin{abstract}
		We prove the main conjecture of \cite[]{geometrization} for integral coefficients in the case of tori and prove that it is \(t\)-exact.
		Along the way we prove that the spectral action as constructed in that manuscript is compatible with the action of the excursion algebra and preserves the grading by \(\pi_1(G)_Q\) on both sides.
		We additionally develop a (non-solidified) version of condensed group (co)homology and show that many constructions from classical group (co)homology extend to that case.
	\end{abstract}
	\maketitle
	\tableofcontents
	\section{Introduction}
Fix a reductive group \(G\) over a non-archimedian local field \(E\) with residue field \(\bbF_q\).
The local Langlands correspondence seeks to find a map of sets 
\begin{equation*}
    \{\text{irreducible smooth representations }G(E)\}\to\{L\text{-parameters }W_E\to\widehat{G}(\bbC)\}
\end{equation*}
satisfying some compatiblity conditions.
Finding the correct conditions to pin down the map is still an open area of research.
Recently Laurent Fargues and Peter Scholze have published the manuscript \cite[]{geometrization}.
They develop a sheaf theory for the stack of \(G\)-bundles on the Fargues-Fontaine curve \(\bun{G}\) for non-torsion coefficients, giving rise to the stable \(\infty\)-category \(\D_{\lis}(\bun{G})\).
This category contains the derived category of smooth representations of \(G(E)\).
It replaces the set of irreducible smooth representations of \(G(E)\).
On the other side, they consider the moduli space (or rather moduli stack) of \(L\)-parameters \([Z^1(W_E,\widehat{G})/\widehat{G}]\).
The set of \(L\)-parameters is then replaced with the stable \(\infty\)-category \(\D^{b,\qc}_{\coh,\nilp}([Z^1(W_E,\widehat{G})/\widehat{G}])\) which is a certain subcategory of the bounded derived category of coherent sheaves.
Additionally they construct an action of \(\perf([Z^1(W_E,\widehat{G})/\widehat{G}])\) on \(\D_{\lis}(\bun{G},\Lambda)\) called the spectral action using Hecke operators, inspired from geometric Langlands theory.
The compatiblity condition is formulated as an equivariance condition regarding this action.
Up to some small technical details, the main conjecture of \cite[]{geometrization} with all decorations reads as:
\begin{custom}{Conjecture}[{\cite[Conjecture I.10.2.]{geometrization}}]
    Let \(\ell\) be a prime coprime to \(q\), \(\Lambda\) the the ring of integers of a algebraic field extension \(L/\bbQ_{\ell}\) and fix \(\sqrt{q}\in\Lambda\).
    Assume that \(G\) is quasi-split.
    Choose a Whittaker datum consisting of a Borel \(B\subset G\) and a generic character \(\psi\from U(E)\to\Lambda^\times\).
    This gives rise to the Whittaker representation \(\cind_{U(E)}^{G(E)}\psi\), which in turn gives rise to a sheaf \(\mathcal{W}\) on \(\bun{G}\).
    There is an equivalence of stable \(\infty\)-categories 
    \begin{equation*}
        \D(\bun{G},\Lambda)^{\omega}\simeq\D^{b,\qc}_{\coh,\nilp}([Z^1(W_E,\widehat{G})/\widehat{G}])
    \end{equation*}
    that is linear over \(\perf([Z^1(W_E,\widehat{G})/\widehat{G}])\), mapping the structure sheaf \(\oo\) to \(\mathcal{W}\).
\end{custom}
In this paper we prove the conjecture when \(G\) is an arbitrary torus over \(E\). In this case, the choice of \(\sqrt{q}\) is not necessary and we can even prove the conjecture for \(\Lambda=\zl\).
\SkipTocEntry
\section*{Outline of the proof}
When \(G=T\) is a torus it turns out that \(\D^{b,\qc}_{\coh,\nilp}([Z^1(W_E,\widehat{T})/\widehat{T}]=\perf^{\qc}([Z^1(W_E,\widehat{T})/\widehat{T}])\).
This means that the spectral action uniquely pins down the equivalence.
For tori the Whittaker datum is unique and if we call \(\mathcal{W}\) the associated sheaf on \(\bun{T}\) then we need to check that \(-*\mathcal{W}\) is an equivalence.

For \(\D_{\lis}(\bun{T},\zl)^{\omega}\), we observe that \(\bun{T}\cong\coprod_{b\in B(T)}[*/\cond{T(E)}]\), thus we get 
\begin{equation*}
    \D_{\lis}(\bun{T},\zl)^{\omega}\simeq\bigoplus_{b\in B(T)}\D(T(E)\mhyphen\Mod_{\zl})^{\omega}
\end{equation*}.
Meanwhile we have  \(\D(T(E)\mhyphen\Mod_{\zl})^{\omega}=\bigcup_{K\subset T(E)}\perf(\zl[T(E)/K])\) where \(K\) runs over open pro-\(p\) subgroups.

On the other side, if we let \(F/E\) be a finite Galois-extension splitting \(T\) with Galois group \(Q\), then \([Z^1(W_E,\widehat{T})/\widehat{T}]\) is a gerbe over \(\Hom(T(E),\gm)\), the space assigning a ring \(R\) the set of continous maps \(T(E)\to\gm(R)\) (equipping \(R\) with the discrete topology).
This eventually boils down to showing that \(H_1(W_{F/E},X_*(T))\cong T(E)\) as condensed abelian groups, where the left hand side needs to be interpreted as a suitable condensed group homology.
This theory will be developed in \cref{condensed group (co)homology}.
The non-condensed version was proven by Langlands in \cite[]{langlandstori} and our proof will proceed by reducing to the non-condensed version.
These things will be checked in \cref{geometric incarnation of Langlands duality for tori}.
Furthermore one observes that this gerbe is banded by \(B\widehat{T}^Q\).
It follows that \(\perf^{\qc}([Z^1(W_E,\widehat{T})/\widehat{T}])=\bigoplus_{\chi\in X^*(\widehat{T}^Q)}\perf^{\qc}([Z^1(W_E,\widehat{T})/\widehat{T}])_{\chi}\).
Here the subscript means that \(\widehat{T}\) acts through \(\chi\) on the cohomology sheaves.
Additionally 
\begin{equation*}
    \perf^{\qc}([Z^1(W_E,\widehat{T})/\widehat{T}])_{0}=\perf^{\qc}(\Hom(T(E),\gm))=\D(T(E)\mhyphen\Mod_{\zl})^{\omega}
\end{equation*}
This holds more generally for gerbes over algebraic spaces banded by diagonalizale groups, as discussed in \cref{sheaves on gerbes banded by diagonalizable groups}.
As remarked in \cite[Remark X.1.5.]{geometrization}, the spectral action matches up the decomposition on both sides under the identification \(B(T)=\pi_1(T)_{Q}\cong X^*(Z(\widehat{T})^{Q})=X^*(\widehat{T}^Q)\).
This works more generally for any reductive group \(G\) and is discussed in \cref{compatibility spectral action with grading}.

There is a line bundle \(\oo[\chi]\) such that \(\oo[\chi]\otimes -\) induces an equivalence 
\begin{equation*}
    \perf^{\qc}([Z^1(W_E,\widehat{T})/\widehat{T}])_{-\chi}\simeq \perf^{\qc}(\Hom(T(E),\gm))
\end{equation*}
Thus we get \(\bigoplus_{\chi\in X^*(B\widehat{T}^Q)}\perf^{\qc}([Z^1(W_E,\widehat{T})/\widehat{T}])_{\chi}=\bigoplus_{\chi\in X^*(B\widehat{T}^Q)}\perf^{\qc}(\Hom(T(E),\gm))\).
Acting with \(\oo[\chi]*-\) induces an equivalence the \(-\chi\)-indexed and \(0\)-indexed summand of \(\D_{\lis}(\bun{T},\zl)^{\omega}\).
Thus we reduce to degree 0.
This is discussed in \cref{reduction to degree 0}.

In \cref{matching of union} we reduce to commutative algebra (meaning symmetric monoidal stable \(\infty\)-categories of the form \(\perf(R)\) for \(R\) some \(E_{\infty}\)-ring spectrum).
Namely a general computation shows that 
\begin{equation*}
    \perf^{\qc}(\Hom(T(E),\gm))\simeq\D(T(E)\mhyphen\Mod_{\zl})^{\omega}
\end{equation*}
\(\Hom(T(E),\gm)\) is a filtered union of open and closed affine subschemes \(\Hom(T(E)/K,\gm)\) with \(K\) running over compact open pro-\(p\) groups of \(T(E)\).
Therefore we can write \(\perf^{\qc}(\Hom(T(E),\gm))\) as a similar filtered union and this matches up with the equality \(\bigcup_{K\subset T(E)}\perf(\zl[T(E)/K])\).
Then one can show that under the spectral action this is compatible with the same filtered union on \(\D_{\lis}(\bun{T}^1,\zl)^{\omega}\).
Additionally, one can check that \(e_K\oo*\mathcal{W}=\zl[T(E)/K]\), where \(e_K\) is the idempotent in \(\oo(Z^1(T(E),\gm))\) cutting out \(Z^1(T(E)/K,\gm)\).
Finally we reduce to the following setting:

We have \(R\) and \(S\) commutative rings (concentrated in degree 0), an exact action \(\perf(R)\) on \(\perf(S)\) and a functor given by \(-*S\).
In \cref{equivalence criterion} we show that such functors are equivalences when they induce isomorphisms on the center.
In \cite[Proposition IX.6.5.]{geometrization} the map from the excursion algebra on the usual Bernstein center \(\mathfrak{Z}(T)\) is computed and one sees that it is an isomorphism.
This map is also constructed using Hecke operators and excursion operators, and the excursion algebra is isomorphic to \(\oo(Z^1(W_E,\widehat{T}))^{\widehat{T}}\).
Thus we finish the proof if we show that the map \(\oo(Z^1(W_E,\widehat{T}))^{\widehat{T}}\to\mathfrak{Z}(T)\) induced by the spectral action agrees with the map constructed via excursion operators.
This is checked in \cref{compatiblity with the excursion algebra}.
\SkipTocEntry
\section*{Acknowledgements}
This work would not have been possible without my advisor Peter Scholze, who suggested this problem to me and whose advice was invaluable.
I would also like to thank Johannes Ansch\"utz for vastly improving the exposition and structure of this paper.
I like to thank both of them for reading preliminaries versions of this paper and pointing out the mistakes in them.
Additionally I am thankful for the Max Planck Institute for Mathematics for their hospitiality.
The author was supported by the DFG Leibniz Preis through Peter Scholze.
\SkipTocEntry
\subsection*{Notation}
\begin{itemize}
    \item \(E\) a non-archimedian local field with residue field \(\bbF_q\). 
    \item \(\ell\) a prime number coprime to \(q\)
    \item \(G\) a reductive group over \(E\), \(T\) a torus over \(E\)
    \item \(F\) a finite Galois extension of \(E\) such that the action of \(W_F\) on \(\widehat{G}\) or \(\widehat{T}\) (it will be clear from the context which situation is discussed) is trivial 
    \item \(Q\defined \mathrm{Gal}(F/E)\) the Galois group of the field extension \(F/E\)
    \item \(C=\widehat{\overline{E}}\) a completed algebraic closure of \(E\), with ring of integers \(\oo_{C}\)
\end{itemize}
Additionally for a stable \(\infty\)-category \(\mathcal{C}\) we will use \(\bcenter(\mathcal{C})\) for \(\End(\id_{\mathcal{C}})\) when considered as an \(E_1\)-ring spectrum.
Taking \(\pi_0\) recovers the usage in \cite[]{geometrization}.
	\section{Recollections from {\cite[]{geometrization}}}
Let \(\Lambda\) be a \(\zl\)-algebra.
\subsection{The stack of \texorpdfstring{$G$}{G}-bundles and \texorpdfstring{$\D_{\lis}(\bun{G},\Lambda)$}{Dlis(BunG,Lambda)}}
\begin{defn}
    Let \(\bun{G}\) be the stack on the category \(\perf_{\overline{\bbF_q}}\) of perfectoid spaces over \(\overline{\bbF_q}\) that sends a perfectoid space \(S\) to the groupoid of \(G\)-bundles on the Fargues-Fontaine curve \(X_S\).
\end{defn}
\begin{thm}[{\cite[Theorem 1.1.]{viehmann}}]
    We have a natural homeomorphism \(|B(G)|\cong |\bun{G}|\) where \(B(G)\) denotes the Kottwitz set topologized via the map \(B(G)\to (X_*(T)_{\bbQ}^+)^Q\times\pi_1(G)_{Q}\) given by the Newton point and the Kottwitz point, where the first factor gets the topology induced by the opposite of the natural poset structure and the second factor the discrete topology.
\end{thm}
\begin{cor}[{\cite[Corollary IV.1.23.]{geometrization}}]
    The Newton point induces a bijection 
    \begin{equation*}
        \kappa\from\pi_0|\bun{G}|\xto{\cong}\pi_1(G)_{Q}
    \end{equation*}
\end{cor}
Given \(b\in B(G)\) let us write \(\bun{G}^b\) for \(\{b\}\times_{|\bun{G}|}\bun{G}\).
Also let \(G_b\) denote the automorphisms of the \(G\)-isocrystal attached to \(b\).
If \(G\) is quasi-split \(G_b\) is an inner form of a Levi subgroup of \(G\) and an inner form of \(G\) if and only if \(b\) is basic.
Generally \(G_b\) is an inner form of a Levi of the quasi-split inner form of \(G\).
Regarding the geometry of the semistable locus, corresponding to the basic \(b\in B(G)\), we have 
\begin{thm}[{\cite[Theorem III.4.5.]{geometrization}}]
    The semistable locus 
    \begin{equation*}
        \bun{G}^{\mathrm{ss}}\subset\bun{G}
    \end{equation*}
    is an open substack and decomposes into open and closed substacks 
    \begin{equation*}
        \bun{G}^{\mathrm{ss}}=\coprod_{b\in B(G)_{\mathrm{basic}}}\bun{G}^b
    \end{equation*}
    and for basic \(b\) we have 
    \begin{equation*}
        [*/\cond{G_b(E)}]\cong\bun{G}^b
    \end{equation*}
\end{thm}
For non-basic \(b\in B(G)\) we still get 
\begin{thm}[{\cite[Proposition III.5.1.]{geometrization}},{\cite[Proposition III.5.3.]{geometrization}}]
    We have 
    \begin{equation*}
        \bun{G}^b=[*/\widetilde{G}_b]
    \end{equation*}
    where \(\widetilde{G}_b\) is a group diamond which is an extension of a ``unipotent'' group diamond with \(\pi_0\widetilde{G}_b\cong G_b(E)\).
\end{thm}
One big difficulty is to develop a formalism of \(\D(\bun{G},\Lambda)\) where \(\Lambda\) is not torsion.
For this Fargues and Scholze define categories of lisse and solid sheaves \(\D_{\lis}(\bun{G},\Lambda)\subset \D_{\solid}(\bun{G},\Lambda)\) and more generally for Artin v-stacks.
Solid sheaves admit a 5-functor formalism, with \(f_{\natural}\) taking the role of \(Rf_!\). 
See \cite[Section VII.3.]{geometrization} for further details on \(f_{\natural}\) and \cite[Chapter VII]{geometrization} for a discussion of the solid formalism in general.
For lisse sheaves we have:
\begin{prop}[{\cite[Proposition VII.7.1.]{geometrization}}]
    There is an equivalence of stable \(\infty\)-categories 
    \begin{equation*}
        \D_{\lis}(\bun{G}^b,\Lambda)\simeq\D(G_b(E)\mhyphen\Mod_{\Lambda})
    \end{equation*}
\end{prop}
\begin{prop}[{\cite[Proposition VII.7.3.]{geometrization}}]
    We have a(n infinite) semi-orthogonal decomposition of \(\D_{\lis}(\bun{G},\Lambda)\) into \(\D_{\lis}(\bun{G}^b,\Lambda)\simeq\D(G_b(E)\mhyphen\Mod_{\Lambda})\).
    In particular we have 
    \begin{equation*}
        \D_{\lis}(\bun{G},\Lambda)\simeq\prod_{\alpha\in \pi_1(G)_{Q}}\D_{\lis}(\bun{G}^{c_1=\alpha},\Lambda)
    \end{equation*}
    Additionally, pullback induces in equivalence 
    \begin{equation*}
        \D_{\lis}(\bun{G},\Lambda)\simeq\D_{\lis}(\bun{G}\times\geompt,\Lambda)
    \end{equation*}
\end{prop}
\begin{cor}\label{cor:vanishing on components by vanishing on pullback}
    Let \(s_b\from\geompt\to\bun{G}\times\geompt\) denote the map corresponding to some geometric point \(b\in B(G)\cong|\bun{G}|\).
    Then the restriction of some \(A\in\D_{\lis}(\bun{G}\times\geompt,\Lambda)\) to a connected component \(\bun{G}^{c_1=\alpha}\) vanishes if and only if \(s_b^*A=0\) for all \(b\in B(G)\) such that \(\kappa(b)=\alpha\).
\end{cor}
Additionally we have compact generation:
\begin{lem}[{\cite[Proposition VII.7.4.]{geometrization}}]
    The stable \(\infty\)-category \(\D_{\lis}(\bun{G},\Lambda)\) is compactly generated.
    Let \(i^b\from\bun{G}^b\to\bun{G}\) denote the inclusion.
    An object \(A\) lies in the subcategory of compact objects \(\D_{\lis}(\bun{G},\Lambda)^{\omega}\) if and only if it has finite support and \(i^{b*}A\in\D_{\lis}(\bun{G}^b,\Lambda)\simeq\D(G_b(E),\Lambda)\) is compact, i.e. lies in the thick stable subcategory generated by \(\cind_K^{G_b(E)}\Lambda\) for compact open pro-\(p\) subgroups \(K\subset G_b(E)\).
\end{lem}
\begin{rem}
    For a torus \(T\) over \(E\) we have \(B(T)=\pi_1(T)_{Q}=X_*(T)_{Q}\).
    In particular all \(b\in B(T)\) are basic and we get 
    \begin{equation*}
        \bun{T}=\coprod_{b\in B(T)}[*/\cond{T_b(E)}]
    \end{equation*}
    We have \(T_b(E)\cong T(E)\), so 
    \begin{equation*}
        \D_{\lis}(\bun{T},\Lambda)=\prod_{[\chi]\in X_*(T)_{Q}}\D(T(E)\mhyphen\Mod_{\zl})
    \end{equation*}
\end{rem}
\subsection{Geometric Satake and Hecke operators}
\begin{defn}
    Let \(\Div^d_{X}\defined(\spd(E)/\varphi^{\bbZ})^d/\Sigma_d\).
    \cite[Proposition VI.1.2.]{geometrization} tells us that if \(S\) is a perfectoid space over \(\overline{\bbF_q}\) a map \(S\to\Div^d_{X}\) gives rise to a closed Cartier divisor \(D_S\) on \(X_S\) that is affinoid after localizing in the analytic topology on \(S\). 
    Let \(\mathcal{I}_S\) denote the ideal sheaf corresponding to \(D_S\).
    Assuming \(D_S\) is affinoid let \(B^+_{\Div^d_{X}}(S)\) be the completion of \(\oo_{X_S}\) at \(\mathcal{I}_S\) and let \(B_{\Div^d_X}(S)=B^+_{\Div^d_{X}}(S)[\frac{1}{\mathcal{I}_S}]\).
    Using these we can define the analogues of (positive) loop groups via 
    \begin{equation*}
        L^+_{\Div^d_X}G(S)\defined G(B^+_{\Div^d_{X}}(S)),\qquad L_{\Div^d_X}G(S)\defined G(B_{\Div^d_X}(S))
    \end{equation*}
    Define the local Hecke stack to be
    \begin{equation*}
        \localheck_{G,\Div^d_X}\defined L^+_{\Div^d_X}G\backslash L_{\Div^d_X}G/L^+_{\Div^d_X}G
    \end{equation*}
    where the quotient is taken as \'etale stacks.
    Equivalently these are pairs \(G\)-bundles \(\mathcal{E}_1\) and \(\mathcal{E}_2\) on \(\spec(B_{\Div^d_X}^+(S))\) with an isomorphism between them on \(\spec(B_{\Div^d_X}(S))\).
\end{defn}
Geometric Satake should give us an equivalence of some kind of perverse sheaves on the local Hecke stack and representations of the Langlands dual group.
More precisely we have 
\begin{thm}[{\cite[Theorem VI.11.1.]{geometrization}}]
    For finite sets \(I\) with \(|I|=d\) define \begin{equation*}
        \localheck_{G}^I\defined\localheck
    _{G,\Div^d_X}\times_{\Div^d_X}(\Div^1_X)^I
    \end{equation*}
    Then we can define a certain subcategory (namely the flat-perverse ULA sheaves) 
    \begin{equation*}
        \mathrm{Sat}_G^I(\Lambda)\subset\D_{\et}(\localheck_G^I,\Lambda)
    \end{equation*}
    We can give it a symmetric monoidal structure (the fusion product, which enhances the convolution product from a monoidal to a symmetric monoidal structure) such that
    \begin{equation*}
        \mathrm{Sat}_G^I(\Lambda)\simeq \rep_{W_E\mhyphen\Mod_\Lambda}(\widecheck{G}^I)
    \end{equation*}
    as monoidal categories, where \(\widecheck{G}\) is some group scheme already defined over \(\widehat{\bbZ}^p\) together with a continuous \(W_E\)-action and we are taking representations on finitely generated projective \(\Lambda\)-modules with continuous \(W_E\)-action.
    After base changing to \(\widehat{\bbZ}^p[\sqrt{q}]\) the group \(\widecheck{G}\) becomes \(W_E\)-equivariantly isomorphic to the usual Langlands dual group \(\widehat{G}\).
\end{thm}
\begin{rem}
    The isomorphism to the Langlands dual group generally only exists \(W_E\)-equivariantly after adjoining \(\sqrt{q}\) because the action on the pinning, that is isomorphisms \(\psi_a\from \mathrm{Lie}(\widehat{U}_a)\cong\widehat{\bbZ}^p\) for each simple root \(a\), is twisted by a Tate twist.
    For tori there are no roots, so the isomorphism exists before adjoining \(\sqrt{q}\) and we can obtain all results for \(\zl\) already.
    See also \cite[Lemma 5.5.7.]{affine_grassmannian_intro} for further discussions.
\end{rem}
\begin{rem}\label{rem: satake category embedding in lisse}
    Let \(\bbD\) denote the Verdier duality functor on \(\D_{\et}(\localheck^I_G,\Lambda)\) for the morphism 
    \begin{equation*}
        \localheck^I_G\to[(\Div^1_X)^I/L_{(\Div^1_X)^I}G]
    \end{equation*}
    Then we get an embedding
    \begin{equation*}
        \sat_G^I(\Lambda)\injto \D_{\solid}(\localheck^I_G,\Lambda),\qquad A\mapsto\rshhom_{\D_{\solid}(\localheck^I,\Lambda)}(\bbD(A),\Lambda)
    \end{equation*}
    that works well for using objects in the Satake category as convolution kernels.
\end{rem}
\begin{defn}\label{defn: hecke action}
    Let \(\heck^I_G\) parametrize maps \(S\to(\Div^1_X)^I\) giving rise to a closed Cartier divisor \(D_S\) and a pair of \(G\)-bundles \(\mathcal{E}_1\), \(\mathcal{E}_2\) and a meromophic modification \(\mathcal{E}_1|_{X_S\setminus D_S}\xto{\cong}\mathcal{E}_2|_{X_S\setminus D_S}\).
    There is a natural map \(\lheck\from\heck^I_G\to\bun{G}\) picking out the left bundle, \(\rheck\from\heck^I_G\to\bun{G}\times(\Div^1_X)^I\) picking out the right bundle and the map \(S\to(\Div^1_X)^I\) and \(q\from\heck^I_G\to\localheck^I_G\) that formally completes at \(D_S\).
    These fit into the following diagram:
    \begin{equation*}
        \begin{tikzcd}
            & \heck^I_G \arrow[ld, "\lheck"'] \arrow[rd, "\rheck"] \arrow[r, "q"] & \localheck^I_G            \\
    \bun{G} &                                                                     & \bun{G}\times(\Div^1_X)^I
    \end{tikzcd}
    \end{equation*}
    Let \(\mathcal{S}_V\) denote the solid sheaf on \(\localheck^I_G\) attached to a representation \(V\in\rep_{\Lambda}((\widehat{G}\rtimes Q)^I)\) via the embedding \cref{rem: satake category embedding in lisse}.
    Then \(T_V(A)\defined \rheck_{\natural}(\lheck^*\lsotimes_{\Lambda}q^*\mathcal{S}_V)\) defines a functor 
    \begin{equation*}
        T_V\from\D_{\solid}(\bun{G},\Lambda)\to\D_{\solid}(\bun{G}\times(\Div^1_X)^I,\Lambda)
    \end{equation*}
    One can show that pullback using \((\Div^1_X)^I\to[*/\cond{W_E}]^I\) embeds \(\D_{\solid}(\bun{G}\times [*/\cond{W_E}]^I,\Lambda)\simeq\D_{\solid}(\bun{G},\Lambda)^{BW_E^I}\) as a full subcategory of \(\D_{\solid}(\bun{G}\times(\Div^1_X)^I,\Lambda)\) via pullback and that \(T_V\) restricts to
    \begin{equation*}
        T_V\from\D_{\lis}(\bun{G},\Lambda)\to\D_{\lis}(\bun{G},\Lambda)^{BW_E^I}
    \end{equation*}
    This is the \defword{Hecke operator} attached to \(V\).
\end{defn}
\begin{rem}
    Hecke operators give rise to exact \(\rep_{\Lambda}(Q^I)\)-linear monoidal functors 
    \begin{equation*}
        \rep_{\Lambda}((\widehat{G}\rtimes Q)^I)\to\End_{\Lambda}(\D_{\lis}(\bun{G},\Lambda))^{BW_E^I}
    \end{equation*}
    The formation of these functors is functorial in \(I\) and one can replace \(\D_{\lis}(\bun{G},\Lambda)\) by \(\D_{\lis}(\bun{G},\Lambda)^{\omega}\), the subcategory spanned by the compact objects.
\end{rem}
\subsection{The stack of \texorpdfstring{\(L\)}{L}-parameters and the spectral action}
\begin{prop}[{\cite[Theorem VIII.1.3.]{geometrization}}]
    There is a scheme \(Z^1(W_E,\widehat{G})\)  over \(\zl\) whose \(R\)-points parametrize condensed 1-cocycles \(W_E\to\widehat{G}(R)\).
    The condensed structure on \(R\) is given by \(R_{\mathrm{disc}}\otimes_{\zld}\zl\).
    It is a disjoint union of affine schemes and can also be written as an increasing union of \(Z^1(W_E/P,\widehat{G})\), where \(P\) is running over open pro-\(p\) groups of the wild inertia such that the action of \(W_E\) on \(\widehat{G}\) factors over \(W_E/P\).
\end{prop}
\begin{proof}[Proof Sketch]
    One checks that any condensed 1-cocycle will be trivial on an open subgroup of the wild inertia.
    This shows that \(Z^1(W_E,\widehat{G})\) is a union of the \(Z^1(W_E/P,\widehat{G})\).
    Then, using that \(\widehat{G}(R)\) is quasi-separated, we may discretize the tame inertia in \(W_E/P\) to obtain a discrete subgroup \(W\) surjecting to \(Q\) satisfying \(Z^1(W_E/P,\widehat{G})\cong Z^1(W,\widehat{G})\).
\end{proof}
\begin{defn}
    The \defword{stack of \(L\)-parameters} is the quotient stack \([Z^1(W_E,\widehat{G})/\widehat{G}]\) where \(\widehat{G}\) acts on a condensed 1-cocycle by conjugation.
\end{defn}
\begin{defn}
    The \defword{spectral Bernstein center} is \(\mathcal{Z}^{\mathrm{spec}}(G,\Lambda)\defined\oo(Z^1(W_E,\widehat{G})^{\widehat{G}})\) where \(\widehat{G}\) acts on a 1-cocycle by conjugation.
    Equivalently these are global functions on \([Z^1(W_E,\widehat{G})/\widehat{G}]\).
\end{defn}
\begin{defn}
    Let \(W\) be a group (usually a discretization of \(W_E/P\) for some \(P\)).
    Then the \defword{algebra of excursion operators} is 
    \begin{equation*}
        \mathrm{Exc}(W,\widehat{G})=\colim_{(n,F_n\to W)}\oo(Z^1(F_n,\widehat{G}))^{\widehat{G}}
    \end{equation*}
    for \(F_n\) denoting the free group on \(n\) elements.
    Let \(\mathrm{Exc}(W,\widehat{G})_{\Lambda}\) denote the base change to \(\Lambda\).
\end{defn}
\begin{thm}[{\cite[Theorem VIII.5.1.]{geometrization}}]
    Assume that $\ell$ does not divide \(|\pi_0(Z(G))|\). Then the map
    \begin{equation*}
        \colim_{(n,F_n\to W)} \oo(Z^1(F_n,\hat{G}))\to \oo(Z^1(W,\hat{G}))
    \end{equation*}
    is an isomorphism in the \(\infty\)-category 
    \(\mathrm{Ind}\perf(B\widehat{G})\). Moreover, the \(\infty\)-category \(\perf([Z^1(W,\widehat{G})/\widehat{G}])\) is generated under cones and retracts by \(\perf(B\widehat{G})\), and \(\mathrm{Ind}\perf([Z^1(W,\widehat{G})/\widehat{G}])\) identifies with the \(\infty\)-category of modules over \(\oo(Z^1(W,\hat{G}))\) in \(\mathrm{Ind}\perf(B\widehat{G})\).
\end{thm}
Recall that we had Hecke operators inducing \(\rep_{\Lambda}(Q^I)\)-linear monoidal functors 
\begin{equation*}
    \rep_{\Lambda}((\widehat{G}\rtimes Q)^I)\to\End_{\Lambda}(\D_{\lis}(\bun{G},\Lambda)^{\omega})^{BW_E^I}
\end{equation*}
functorial in \(I\).
\begin{prop}[{\cite[Proposition IX.5.1.]{geometrization}}]
    Letting \(P\) run over the open pro-\(p\)-subgroups of the wild inertia, we have
    \begin{equation*}
        \D_{\lis}(\bun{G},\Lambda)^{\omega}=\bigcup_{P\subset W_E}\D_{\lis}^P(\bun{G},\Lambda)^{\omega}
    \end{equation*}
    where \(\D_{\lis}^P(\bun{G},\Lambda)^{\omega}\) denotes the subcategory of \(\D_{\lis}(\bun{G},\Lambda)^{\omega}\) consisting of those objects \(A\) such that for all Hecke operators \(T_V\) the \(W_E^I\)-action on \(T_VA\) factors through \((W_E/P)^I\).
    Letting \(W\) be a subgroup of \(W_E/P\) discretizing the tame inertia we get 
    \begin{equation*}
        (\D_{\lis}^P(\bun{G},\Lambda)^{\omega})^{B(W_E/P)^I}\simeq (\D_{\lis}^P(\bun{G},\Lambda)^{\omega})^{BW^I}
    \end{equation*}
\end{prop}
Thus the Hecke operators give us \(\rep_{\Lambda}(Q^I)\)-linear exact monoidal functors 
\begin{equation*}
    \rep_{\Lambda}((\widehat{G}\rtimes Q)^I)\to\End_{\Lambda}(\D_{\lis}^P(\bun{G},\Lambda)^{\omega})^{BW^I}
\end{equation*}
functorial in finite sets \(I\).
Generally, replacing \(\D_{\lis}^P(\bun{G},\Lambda)^{\omega}\) with an idempotent complete stable \(\infty\)-category \(\mathcal{C}\), we get an action of the excursion algebra on \(\mathcal{C}\) and an action of \(\perf([Z^1(W,\widehat{G})/\widehat{G}])\) on \(\mathcal{C}\). 
More precisely:
\begin{thm}[{\cite[Theorem VIII.4.1.]{geometrization}}]
    Let \(\mathcal{C}\) be a small stable \(\Lambda\)-linear \(\infty\)-category.
    The data of \(\rep_{\Lambda}(Q^I)\)-linear functors 
    \begin{equation*}
        \rep_{\Lambda}((\widehat{G}\rtimes Q)^I)\to\End_{\Lambda}(\mathcal{C})^{BW^I}
    \end{equation*}
    functorial in \(I\) such that \(T_{V\otimes W}\cong T_V\circ T_W\) gives rise to a morphism of rings 
    \begin{equation*}
        \exc(W,\widehat{G})\to\pi_0\End(\id_{\mathcal{C}})
    \end{equation*}
\end{thm}
\begin{proof}[Proof Sketch]
    By definition it suffices to construct maps \(\oo(Z^1(F_n,\widehat{G}))^{\widehat{G}}\to\pi_0\End(\id_{\mathcal{C}})\) for each \((n,F_n\to W_E)\) functorial in such pairs.
    Given \(f\in\oo(\widehat{G}\backslash (\widehat{G}\rtimes Q)^I/\widehat{G})\), we consider the \((\widehat{G}\rtimes Q)^I\)-representation \(V_f\) that is the subrepresentation of \(\oo((\widehat{G}\rtimes Q)^I/\widehat{G})\) generated by \(f\).
    There are natural maps \(\alpha\from 1\to V_f|_{\rep_{\Lambda}(\widehat{G})}\) induced by \(f\) and \(\beta\from V_f|_{\rep_{\lambda}(\widehat{G})}\to 1\) given by evaluation at the identity.
    If we are given \((g_i)_{i\in I}\in W^I\), we get an element in \(\pi_0\End(\id_(\mathcal{C}))\) by considering the natural transformation 
    \begin{equation*}
        \id_{\mathcal{C}}=T_1\xto{\alpha}T_{V_f}\xto{(g_i)_{i\in I}}T_{V_f}\xto{\beta}T_1=\id_{\mathcal{C}}
    \end{equation*}
    This is the same as giving maps of \(\Lambda\)-modules 
    \begin{equation*}
        \Theta^I\from \oo(\widehat{G}\backslash(\widehat{G}\rtimes Q)^I/\widehat{G})\to \mathrm{Maps}(W^I,\pi_0\End_{\Lambda}(\mathcal{C}))
    \end{equation*}
    These are functorial in \(I\) (since the Hecke operators are) and are algebra maps (since \(T_{V\otimes W}\cong T_V\circ T_W\)).
    Then one identifies
    \begin{equation*}
        \oo(\widehat{G}\backslash(\widehat{G}\rtimes Q)^{\{0,\dots,n\}}/\widehat{G})\otimes_{\oo(Q^{\{0,\dots,n\}})}\oo(Q^{\{1,\dots,n\}})\cong\oo((\widehat{G}\rtimes Q)^{n}/\!\!/\widehat{G})
    \end{equation*}
    via pullback \((g_1,\dots,g_n)\mapsto (1,g_1,\dots,g_n)\).
    One has a similar story for \(\mathrm{Maps}(W^{\{0,\dots,n\}},\pi_0\End_{\Lambda}(\id_{\mathcal{C}}))\) and \(\mathrm{Maps}(W^n,\pi_0\End_{\Lambda}(\id_{\mathcal{C}}))\).
    Thus one gets algebra maps 
    \begin{equation*}
        \Theta_n\from \oo((\widehat{G}\rtimes Q)^{n}/\!\!/\widehat{G})\to \mathrm{Maps}(W^n,\pi_0\End_{\Lambda}(\id_{\mathcal{C}}))
    \end{equation*}
    Such maps are equivalent to maps \(\oo(Z^1(F_n,\widehat{G}))^{\widehat{G}}\to\pi_0\End_{\Lambda}(\id_{\mathcal{C}})\).
\end{proof}
\begin{thm}[{\cite[Theorem X.0.2.]{geometrization}}]
    Let \(\mathcal{C}\) be a small stable idempotent-complete \(\Lambda\)-linear \(\infty\)-category.
    Assume that \(\ell\) is invertible in \(\Lambda\) or that \(\ell\) does not divide \(|\pi_0(Z(G))|\).
    The anima of \(\rep_{\Lambda}(Q^I)\)-linear exact monoidal functors 
    \begin{equation*}
        \rep_{\Lambda}((\widehat{G}\rtimes Q)^I)\to\End_{\Lambda}(\mathcal{C})^{BW^I}
    \end{equation*}
    functorial in \(I\) is equivalent to the anima of \(\Lambda\)-linear actions of \(\perf([Z^1(W,\widehat{G})/\widehat{G}])\) on \(\mathcal{C}\).
    More precisely the action of \(W\) on the projection map \(p\from[Z^1(W,\widehat{G})/\widehat{G}]\cong\mathrm{Maps}_{BQ}(BW,B(\widehat{G}\rtimes Q))\to B(\widehat{G}\rtimes Q)\)
    gives rise to a functor \(p^*\from\perf(B(\widehat{G}\rtimes Q))\to\perf([Z^1(W,\widehat{G})/\widehat{G}])^{BW}\) that we can extend to functors 
    \begin{equation*}
        p^*_I\from\perf(B(\widehat{G}\rtimes Q)^I)\to\perf([Z^1(W,\widehat{G})/\widehat{G}])^{BW^I}
    \end{equation*}
    via tensor products.
    Then giving a \(\Lambda\)-linear action of  \(\perf([Z^1(W,\widehat{G})/\widehat{G}])\) on \(\mathcal{C}\) 
    \begin{equation*}
        \perf([Z^1(W,\widehat{G})/\widehat{G}]) \to\End_{\Lambda}(\mathcal{C})
    \end{equation*}
    gives rise to 
    \begin{equation*}
        \rep_{\Lambda}((\widehat{G}\rtimes Q)^I)\to\End_{\Lambda}(\mathcal{C})^{BW^I}
    \end{equation*}
    by precomposing with \(p_I^*\).
\end{thm}
\begin{proof}[Proof Sketch]
    Let \(S\to\perf(\mathrm{Maps}_{BQ}^{\Sigma}(S,B(\widehat{G}\rtimes Q)))\) be the extension of the functor 
    \begin{equation*}
        I\to \perf(\mathrm{Maps}_{BQ}(I,B(\widehat{G}\rtimes Q)))
    \end{equation*} from finite sets \(I\) over \(BQ\) to \(\Lambda\)-linear symmetric monoidal small stable idempotent-complete \(\infty\)-categories to all anima via sifted colimits.
    Then this functor commutes with sifted colimits in \(S\), and using Yoneda's lemma and adjunctions we see that the datum of \(\rep_{\Lambda}(Q^I)\)-linear exact monoidal functors 
    \begin{equation*}
        \rep_{\Lambda}((\widehat{G}\rtimes Q)^I)\to\End_{\Lambda}(\mathcal{C})^{S^I}
    \end{equation*}
    functorial in \(I\) is equivalent to \(\Lambda\)-linear exact actions of \(\perf(\mathrm{Maps}_{BQ}^{\Sigma}(S,B(\widehat{G}\rtimes Q)))\) on \(\mathcal{C}\).
    There is a natural functor 
    \begin{equation*}
        \perf(\mathrm{Maps}_{BQ}^{\Sigma}(S,B(\widehat{G}\rtimes Q)))\to \perf(\mathrm{Maps}_{BQ}(S,B(\widehat{G}\rtimes Q)))
    \end{equation*}
    that is an embedding for \(S=BF_n\) (\cite[Proposition X.3.3.]{geometrization}) and an equivalence for \(S=BW\) (\cite[Proposition X.3.4.]{geometrization}, this is where we need \(\ell\) not divide \(|\pi_0(Z(G))|\) or that \(\ell\) is invertible in \(\Lambda\)).
\end{proof}
\begin{rem}
    To formulate the main conjecture of \cite[]{geometrization} we need to impose a certain support condition called ``nilpotent singular support'' on the coherent sheaves on \([Z^1(W_E,\widehat{G})/\widehat{G}]\).
    This is discussed in \cite[Section VIII.2.2.]{geometrization}.
    For the purposes of this paper note that the nilpotent cone for tori is \(\{0\}\subset\mathfrak{t}^*\), thus by \cite[Theorem VIII.2.9]{geometrization} we have 
    \begin{equation*}
        \D_{\coh,\nilp}^{b,\qc}([Z^1(W_E,\widehat{T})/\widehat{T}])\simeq\perf^{\qc}([Z^1(W_E,\widehat{T})/\widehat{T}]).
    \end{equation*}
\end{rem}

	\section{Condensed group (co)homology}\label{condensed group (co)homology}
In this subsection we develop the necessary amount of condensed group (co)homology to formulate what transfer is in the condensed setting.
For this entire subsection, let \(G\) be a condensed group, \(H\) a condensed subgroup, \(M\) a condensed abelian group with a linear \(G\)-action and \(S\) an extremally disconnected set.
\begin{defn}
    Define the \defword{condensed group homology} respectively \defword{condensed group cohomology} to be 
    \begin{align*}
        H_i(G,M)&\defined\pi_i(\bbZ\lotimes_{\bbZ[G]}M)\\
        \shortintertext{respectively}
        H^i(G,M)&\defined\pi_{-i}\riHom_{\bbZ[G]}(\bbZ,M)
    \end{align*}
    where \(\bbZ[G]\) is the sheafification of the presheaf \(S\mapsto \bbZ[G(S)]\).
    We will also write \(M_G\) for \(H_0(G,M)\) and \(M^G\) for \(H^0(G,M)\).
\end{defn}
\begin{rem}
    Note that 
    \begin{align*}
        \bbZ\lotimes_{\bbZ[H]}M&\cong \bbZ[G/H]\lotimes_{\bbZ[G]}M\\    
        \shortintertext{and}
        \riHom_{\bbZ[H]}(\bbZ,M)&\cong\riHom_{\bbZ[G]}(\bbZ[G/H],M)
    \end{align*}
    In particular, if \(G/H\) is a condensed group, the condensed group (co)homology carries a residual \(G/H\)-action.
\end{rem}
\begin{rem}\label{rem: underlying set condensed homology is ordinary homology}
    Since evaluation on extremally disconnected sets is strong monoidal and exact we see that
    \begin{equation*}
        (\bbZ\lotimes_{\bbZ[G]}M)(S)=\bbZ(S)\lotimes_{\bbZ[G](S)}M(S)
    \end{equation*}
    so in particular \(H_i(G,M)(*)=H_i(G(*),M(*))\).
\end{rem}
\begin{constr}
    Let \(P_\bullet(S)\to\bbZ(S)\) be the bar resolution for \(G(S)\).
    Then the sheafification of \(S\mapsto P_\bullet(S)\) is a flat resolution of \(\bbZ\) with \(P_n=\bbZ[G][G^n]\), which we will also call the bar resolution.
\end{constr}
\begin{rem}\label{rem: bar resolution not projective}
    The bar resolution is not generally a projective resolution.
    Therefore it is generally unsuited to compute condensed group cohomology.
\end{rem}
\begin{rem}
    If \(G\) is a discrete abstract group, then the bar resolution is just the ordinary bar resolution with the discrete condensed structure.
    It follows that 
    \begin{equation*}
        H_i(G,M)=\cond{H_i(G(*),M(*))}
    \end{equation*}
    whenever \(M\) is a discrete \(G\)-module.
\end{rem}
\begin{defn}
    Let \(N\) be a condensed \(H\)-module.
    Define the \defword{coinduction} respectively \defword{induction} by
    \begin{align*}
        \ind_H^G(N)\defined \bbZ[G]\otimes_{\bbZ[H]}N
        \shortintertext{respectively}
        \coind_H^G(N)\defined\iHom_{\bbZ[H]}(\bbZ[G],N)
    \end{align*}
    The coinduction is naturally a right \(G\)-module, but we will consider it as a left \(G\)-module by precomposing with the inversion on \(G\).
\end{defn}
\begin{lem}
    We have adjunctions
    \begin{equation*}
        \ind_H^G\dashv\res_H^G\dashv\coind_H^G
    \end{equation*}
\end{lem}
\begin{proof}
    Clearly there is an adjunction on the presheaf versions of these functors.
    One can just sheafify that.
\end{proof}
\begin{lem}\label{lem: Z[G] as Z[H] module}
    Assume that \(G/H\) is a discrete set.
    A set-theoretic section \(s\) of \(G(*)\to G/H(*)\) gives us isomorphisms \(\bbZ[H][G/H]\cong \bbZ[G]\) and \(\bbZ[H][H\backslash G]\cong \bbZ[G]\) as right respectively left \(\bbZ[H]\)-modules.
\end{lem}
\begin{proof}
    We do the case for right \(\bbZ[H]\)-modules.
    The other case is similar.
    We have \(\bbZ[H][G/H]\cong\bbZ[H\times G/H]\) as right \(\bbZ[H]\)-modules, where \(H\) acts on \(H\times G/H\) naturally on the first factor and trivially on the second.
    Thus we need to check that \(H\times G/H\cong G\) as condensed sets with \(H\)-action.
    By adjunction the map \(s\from G/H(*)\to G(*)\) gives rise to a morphism of condensed sets \(s\from\cond{G/H(*)}\to G\).
    By assumption \(\cond{G/H(*)}=G/H\).
    We define \((H\times G/H)(S)\to G(S)\) by \((h,\overline{g})\to s(\overline{g}h)\) and \(G(S)\to (H\times G/H)(S)\) by \(g\mapsto (s(\overline{g}^{-1})g,\overline{g})\) where \(\overline{g}\) is the image of \(g\) under \(G\to G/H\).
    These are inverse maps and \(H\)-linear.
\end{proof}
\begin{lem}[Shapiro's lemma]
    We have a canonical isomorphism
    \begin{align*}
        H_*(H,N)\xto{\cong} H_*(G,\ind_H^G N)
    \end{align*}
    If \(\bbZ[G]\) is projective over \(\bbZ[H]\) (for example when \(G/H\) comes from a finite discrete set, by \cref{lem: Z[G] as Z[H] module}), then
    \begin{equation*}
        H^*(G,\coind_H^G N)\xto{\cong}H^*(H,N)
    \end{equation*}
\end{lem}
\begin{proof}
    This is clear for homology, since
    \begin{equation*}
        H_*(G,\ind_{H}^G N)\cong H_*(P_{\bullet}\otimes_{\bbZ[G]}\bbZ[G]\otimes_{\bbZ[H]}M)\cong H_*(P_{\bullet}\otimes_{\bbZ[H]}N)
    \end{equation*}
    where \(P_{\bullet}\) is the bar resolution for \(\bbZ\) as a \(\bbZ[G]\)-module.
    Note that \(\bbZ[G]\) is flat as a \(\bbZ[H]\)-module, as this is clearly true on the level of presheaves.
    If \(\bbZ[G]\) is projective over \(\bbZ[H]\), then 
    \begin{equation*}
        \riHom_{\bbZ[G]}(\bbZ,\coind_H^GN)\cong\riHom_{\bbZ[G]}(\bbZ,\riHom_{\bbZ[H]}(\bbZ[G],N))\cong\riHom_{\bbZ[H]}(\bbZ,N)
    \end{equation*}
    which implies the claim for cohomology.
\end{proof}
\begin{lem}\label{lem: ambidexterity condensed group representation}
    Assume that \(G/H\) is a finite discrete set.
    Then we have a \(G\)-equivariant isomorphism
    \begin{equation*}
        \ind_H^G N\cong\coind_H^G N
    \end{equation*}
\end{lem}
\begin{proof}
    We have
    \begin{align*}
        \ind_H^GN&\cong\bbZ[G]\otimes_{\bbZ[H]}N\\
        &\cong\bbZ[G/H]\otimes_{\bbZ} N\\
        &\cong\bigoplus_{s\in G/H}N\\
        &\cong\bigoplus_{s\in H\backslash G}N\\
        &\cong\iHom_{\bbZ}(\bbZ[H\backslash G],N)\\
        &\cong\iHom_{\bbZ[H]}(\bbZ[G],N)
    \end{align*}
    Note that while the third and fifth isomorphism depends on a choice of set-theoretic section to \(G(*)\to G/H(*)\), the entire composite is independent of this choice.
\end{proof}
\begin{defn}\label{def: transfer condensed group (co)homology}
    Assume that \(G/H\) is a finite discrete set.
    We define transfer maps in condensed group (co)homology by setting
    \begin{align*}
        \tr^G_H&\from H_*(G,M)\to H_*(G,\coind^G_H\res^G_HM)\xto{\cong}H_*(G,\ind^G_H\res^G_H M)\xto{\cong}H_*(H,\res^G_H M)\\
        \tr^G_H&\from H^*(H,M)\xto{\cong} H^*(G,\coind^G_H\res^G_HM)\xto{\cong}H^*(G,\ind^G_H\res^G_H M)\to H^*(H,\res^G_H M)
    \end{align*}
    The isomorphisms are Shapiro's lemma and the previous lemma, the other arrows are the unit or counit of the coinduction, restriction, induction adjunction.
\end{defn}
\begin{lem}\label{lem: functoriality of transfer}
    Assume that \(H'\) is a normal subgroup of \(H\) and \(G\).
    Assume that \(M\) comes from a \(G/H'\)-module.
    Then the following diagram commutes:
    \begin{equation*}
        \begin{tikzcd}
            {H_*(G,M)} \arrow[d] \arrow[r, "\tr^G_H"]    & {H_*(H,M)} \arrow[d] \\
            {H_*(G/H',M)} \arrow[r, "\tr_{H/H'}^{G/H'}"] & {H_*(H/H',M)}       
            \end{tikzcd}
    \end{equation*}
\end{lem}
\begin{proof}
    We have maps
    \begin{align*}
        \coind^G_HM=\iHom_{\bbZ[H]}(\bbZ[G],M)\cong\iHom_{\bbZ[H/H']}(\bbZ[G]\otimes_{\bbZ[H]}\bbZ[H/H'],M)=\coind_{H/H'}^{G/H'}M\\
        \shortintertext{and}    
        \ind^G_H M=\bbZ[G]\otimes_{\bbZ[H]}M\to\bbZ[G/H']\otimes_{\bbZ[H]}M\cong\bbZ[G/H']\otimes_{\bbZ[H/H']}M=\ind^{G/H'}_{H/H'}M
    \end{align*}
    where \(\bbZ[G/H']\otimes_{\bbZ[H]}M\cong\bbZ[G/H']\otimes_{\bbZ[H/H']}M\) follows from the corresponding statement on the presheaf level.
    These fit into the following commuting diagram:
    \begin{equation*}
        \begin{tikzcd}[column sep=small]
            {H_*(G,M)} \arrow[d] \arrow[r] & {H_*(G,\coind^G_HM)} \arrow[d] \arrow[r]    & {H_*(G,\ind^G_HM)} \arrow[d] \arrow[r] & {H_*(H,M)} \arrow[d] \\
            {H_*(G/H',M)} \arrow[r]        & {H_*(G/H',\coind^{G/H'}_{H/H'}M)} \arrow[r] & {H_*(G,\ind^{G/H'}_{H/H'}M)} \arrow[r] & {H_*(H/H',M)}       
            \end{tikzcd}
    \end{equation*}
    The top horizontal arrows define \(\tr^G_H\) and the bottom horizontal arrows define \(\tr^{G/H'}_{H/H'}\).
\end{proof}
\begin{rem}\label{rem: transfer and G-action}
    Observe that in the setting of the \cref{def: transfer condensed group (co)homology}, that for \(g\in G(*)\) we get commuting triangles
    \begin{equation*}
        \begin{tikzcd}
            & {H_*(G,M)} \arrow[ld, "\tr^G_H"'] \arrow[rd, "\tr^G_{gHg^{-1}}"] &                   \\
            {H_*(H,M)} \arrow[rr, "g.(-)"] &                                                                & {H_*(gHg^{-1},M)}
        \end{tikzcd}
    \end{equation*}
    \begin{equation*}
        \begin{tikzcd}
            & {H^*(G,M)} &                                                                      \\
            {H^*(H,M)} \arrow[ru, "\tr^G_H"] &            & {H^*(gHg^{-1},M)} \arrow[ll, "g.(-)"'] \arrow[lu, "\tr^G_{gHg^{-1}}"']
            \end{tikzcd}
    \end{equation*}
\end{rem}
\begin{lem}\label{lem: homology inclusion is norm map}
    Assume that \(G'\) is a subgroup of \(G\) such that any quotient of the tower \(H\subset G'\subset G\) is finite discrete and \(H\) is a normal subgroup of \(G'\) and \(G\).
    Then the following diagram commutes:
    \begin{equation*}
        \begin{tikzcd}
            {H_*(G',M)} \arrow[d] \arrow[r, "\tr^{G'}_H"] & {H_*(H,M)} \arrow[d, "\sum_{g\in G/G'(*)}g.(-)"] \\
            {H_*(G,M)} \arrow[r, "\tr^G_H"]               & {H_*(H,M)}                                      
            \end{tikzcd}
    \end{equation*}
    Note the implicit choice of coset representatives for \(G/G'(*)\).
\end{lem}
\begin{proof}
    By dimension shifting we reduce to the case where \(*=0\).
    Then as \(\bbZ[G']\)-modules we have \(\coind^G_H M=\iHom_{\bbZ[H]}(\bbZ[G],M)=\iHom_{\bbZ[H]}(\bbZ[G'][G/G'],M)\), once given a choice of coset representatives \(g_i\in G/G'(*)\).
    Unwinding, 
    we see that the following diagram of \(\bbZ[G']\)-modules commutes:
    \begin{equation*}
        \begin{tikzcd}
            & \coind^{G'}_H M \arrow[dd, "\sum g_i.(-)"] \\
M \arrow[ru] \arrow[rd] &                                            \\
            & \coind^G_HM=\bigoplus_{s\in G/G'(*)} M    
\end{tikzcd}
    \end{equation*}
    From this the claim follows. 
\end{proof}
\begin{rem}
    By functoriality of the Lyndon-Hochschild-Serre spectral sequence we can extend it to condensed group homology and condensed group cohomology.
    Assume that \(H\) is a normal subgroup of \(G\), then we have the following 5-term exact sequence:
    \begin{equation*}
        H_2(G,M)\to H_2(G/H,M_H)\to H_1(H,M)_{G/H}\to H_1(G,M)\to H_1(G/H,M_H)\to 0
    \end{equation*}
\end{rem}
	\section{Sheaves on the stack of Langlands parameters}\label{sheaves on the stack of Langlands parameters}
\subsection{Geometric incarnation of Langlands duality for tori}\label{geometric incarnation of Langlands duality for tori}
\begin{lem}\label{lem: units fpqc locally divisible}
    Let \(\Lambda\) be a ring.
    Then there is a ring \(\Lambda'\) with a faithfully flat ring map \(\Lambda\to\Lambda'\) such that \({\Lambda'}^\times\) is divisible.
\end{lem}
\begin{proof}
    Let \(\Lambda_{i+1}\defined\Lambda_i[a^{1/n}\mid a\in\Lambda_i^\times,n\in\bbN]\) and \(\Lambda_0\defined\Lambda\).
    Then set \(\Lambda'\defined\colim_{i\geq 0}\Lambda_i\).
    Since filtered colimits of rings are computed on the underlying abelian group every unit of \(\Lambda'\) is divisible.
    For all \(i\geq 0\) the ring morphism \(\Lambda_i\to\Lambda_{i+1}\) has a retract as a \(\Lambda\)-module map and makes \(\Lambda_{i+1}\) a free \(\Lambda\)-module, as the map \(\Lambda_i\to\Lambda_{i+1}\) is base changed from \(\bbZ[t_i\mid i\in I]\to\bbZ[t_i^{1/n}\mid i\in I,n\in\bbN]\) for any indexing set \(I\).
    Then \(\bbZ[t_i^{1/n}\mid i\in I,n\in\bbN]\) is a free \(\bbZ[t_i\mid i\in I]\)-module with basis \(\prod_{j\in J}t_j^{m_j/n_j}\) with \(J\subset I\) finite and \(m_j<n_j\).
    The map \(\bbZ[t_i\mid i\in I]\to\bbZ[t_i^{1/n}\mid i\in I,n\in\bbN]\) is injective with image being the subspace spanned by \(1\), so it the inclusion of a direct summand.
    It follows that \(\Lambda'\) and \(\Lambda_{i+1}/\Lambda_i\) are flat \(\Lambda\)-modules. 
    Let \(N\neq 0\) be a \(\Lambda\)-module.
    For faithful flatness we need to check that \(\Lambda'\otimes_{\Lambda}N\neq 0\).
    For this observe that \(\Lambda_i\otimes_{\Lambda}N\to\Lambda_{i+1}\otimes_{\Lambda}N\) is injective, so \(\Lambda'\otimes_{\Lambda}N\) is given by a sequential colimit of non-zero modules with injective transition maps.
    Such colimts are non-zero.
\end{proof}
\begin{lem}\label{lem: Hom(T(E) G_m)}
    Let \(\Hom(T(E),\gm)\) be the functor sending a ring \(R\) to continuous group homomorphisms \(T(E)\to\gm(R)\), where we equip \(R\) with the discrete topology.
    Then this is representable by a scheme.
    It can be written as
    \begin{equation*}
        \Hom(T(E),\gm)=\bigcup_{K\subset T(E)}\Hom(T(E)/K,\gm)
    \end{equation*}
    where \(K\) runs over any system of open pro-\(p\) subgroups forming a neighborhood basis of the identity.
\end{lem}
\begin{proof}
    This is done in the same way as \cite[Theorem VIII.1.3.]{geometrization}.
\end{proof}
\begin{lem}
    The stack \([Z^1(W_E,\widehat{T})/\widehat{T}]\) is a gerbe over an algebraic space.
\end{lem}
\begin{proof}
    Using \cite[Tag 06PB]{stacks} we can easily compute that the inertia stack is 
    \begin{equation*}
        \widehat{T}^Q\times[Z^1(W_E,\widehat{T})/\widehat{T}]\to[Z^1(W_E,\widehat{T})/\widehat{T}]
    \end{equation*}
    where the map is the projection.
    The claim follows from \cite[Tag 06QJ]{stacks}.
\end{proof}
Recall that for an algebraic stack \(\mathcal{X}\) that is a gerbe over some algebraic space \(X\) that \(X\) is unique and is given by the fppf-sheafification of \(\pi_0 \mathcal{X}\) (\cite[Tag 06QD]{stacks}), which is also then the coarse moduli space.
Since algebraic spaces are fpqc-sheaves (\cite[Tag 0APL]{stacks}) one can also consider the fpqc-sheafification.
\begin{lem}\label{lem: geometric incarnation langlands duality tori}
    We have a morphism of stacks 
    \begin{equation*}
        \Psi\from [Z^1(W_{F/E},\widehat{T})/\widehat{T}]\to\Hom(T(E),\gm)
    \end{equation*}
    that exhibits \([Z^1(W_{F/E},\widehat{T})/\widehat{T}]\) as a gerbe over an algebraic space.
    Running over \(P\subset F^\times\) open such that \(\Theta(P)\) is pro-\(p\), it is induced from a map 
    \begin{equation*}
        \Psi_P\from Z^1(W_{F/E}/P,\widehat{T})\to \Hom(T(E)/\Theta(P),\gm)
    \end{equation*}
    that sends a cocycle \(f\from W_{F/E}/P\to\widehat{T}(\Lambda)\) to
    \begin{equation*}
        \sum m_i\otimes [g_i]\mapsto\prod f(g_i)(m_i)
    \end{equation*}
    for \(\sum m_i\otimes [g_i]\in H_1(W_{F/E}/P,X_*(T))\) and we identify \(H_1(W_{F/E}/P,X_*(T))\) with \(T(E)/\Theta(P)\) using \cref{prop: transfer condensed iso}.
    For an alternative perspective, there is a natural pairing 
    \begin{equation*}
        H^1(W_{F/E}/P,\Hom(X_*(T),\Lambda^\times))\times H_1(W_{F/E}/P,X_*(T))\to H_0(W_{F/E}/P,\Lambda^\times)=\Lambda^\times
    \end{equation*}
    and the map above is the one induced by this pairing.
\end{lem}
\begin{proof}    
    The map \(\Psi_P\) induces an isomorphism 
    \begin{equation*}
        \pi_0[Z^1(W_{F/E}/P,\widehat{T})/_p\widehat{T}](\Lambda)\cong \Hom(H_1(W_{F/E},X_*(T)),\gm)(\Lambda)
    \end{equation*}
    when \(\Lambda^\times\) is an injective abelian group and the index \(p\) denotes the pre-stack quotient.
    By \cref{lem: units fpqc locally divisible} the fpqc-sheafification of \(\pi_0[Z^1(W_{F/E}/P,\widehat{T})/_p\widehat{T}]\) is \(\Hom(H_1(W_{F/E},X_*(T)),\gm)\).
    Note that \([Z^1(W_{F/E}/P,\widehat{T})/\widehat{T}]\), which is a priori only the fppf-stackification of the prestack quotient \([Z^1(W_{F/E}/P,\widehat{T})/_p\widehat{T}]\), is actually an fpqc-stack and thus is also the fpqc-stackification.
    This is because fpqc-\(\widehat{T}\)-torsors are already fppf-locally trivial, since \(\widehat{T}\to\spec(\zl)\) is fppf.
    The sheafification of \(\pi_0[Z^1(W_{F/E}/P,\widehat{T})/_p\widehat{T}]\) is isomorphic to the sheafification of \(\pi_0[Z^1(W_{F/E}/P,\widehat{T})/\widehat{T}]\).
    Finally we need to see that \(\Hom(T(E),\gm)=\bigcup_{P\subset F^\times}\Hom(T(E)/\Theta(P),\gm)\), where \(P\subset F^\times\) runs over open subgroups of \(F^\times\) such that \(\Theta(P)\) is open pro-\(p\).
    This follows from \cref{lem: condensed H_1 of quotient is torus quotient}.
\end{proof}
Recall that in the classical Langlands duality for tori, the following is a crucial proposition:
\begin{prop}
    Transfer defines a bijection of abelian groups 
    \begin{equation*}
        H_1(W_{F/E},X_*(T))\xto{\cong}H_1(F^\times,X_*(T))^Q=T(E)
    \end{equation*}
\end{prop}
\begin{proof}
    See \cite[]{langlandstori} for a proof.
\end{proof}
Using the material of \cref{condensed group (co)homology}, we can consider the condensed statements and get the following refinement:
\begin{prop}\label{prop: transfer condensed iso}
    Transfer defines an isomorphism of condensed abelian groups 
    \begin{equation*}
        H_1(W_{F/E},X_*(T))\xto{\cong}H_1(F^\times,X_*(T))^Q=T(E)
    \end{equation*}
\end{prop}
\begin{proof}
    By \cref{rem: transfer and G-action} transfer does land in \(H_1(F^\times,X_*(T))^Q\subset H_1(F^\times,X_*(T))\).
    By \cref{lem: sequential colimit qcqs condensed on underlying set} it suffices to show that both sides come from sequential colimits of compact Hausdorff abelian groups.
    This is clear for \(H_1(F^\times,X_*(T))^{Q}=\cond{T(E)}\).
 
    The inflation-restriction sequence on condensed group homology gives us an exact sequence of condensed abelian groups
    \begin{equation*}
        H_2(Q,X_*(T))\xto{\delta} H_1(F^\times,X_*(T))_{Q}\to H_1(W_{F/E},X_*(T))\to H_1(Q,X_*(T))\to 0
    \end{equation*}
    Note that \(H_1(F^\times,X_*(T))\to H_1(F^\times,X_*(T))_{Q}\) is surjective with kernel \(\sum_{\sigma\in Q}\im(\sigma-1)\), so every term comes from a sequential colimit of compact Hausdorff abelian groups.
    It follows from \cref{lem: sequential colimit Hausdorff closed under extension finite (co)limits} that then \(H_1(W_{F/E},X_*(T))\) is also such a colimit.
\end{proof}
\begin{lem}\label{lem: Z[S] pseudo-coherent}
    Let \(S\) be a compact Hausdorff space.
    Then \(\bbZ[S]\) is a pseudo-coherent condensed abelian group.
\end{lem}
\begin{proof}
    Let \(M=\colim_{i\in I}M_i\) be a filtered colimit of condensed abelian groups.
    Fix some uncountable strong limit cardinal \(\kappa\) such that all \(M_i\), \(M\) and \(\bbZ[S]\) are \(\kappa\)-condensed.
    To check that 
    \begin{equation*}
        \mathrm{Ext}^n(\bbZ[S],M)=\colim_{i\in I}\mathrm{Ext}^n(\bbZ[S],M_i)
    \end{equation*}
    note that these can be computed in \(\kappa\)-condensed sets.
    The claim follows from the fact that \(\kappa\)-condensed sets from a coherent topos (\(\kappa\)-small compact Hausdorff spaces forming a family of qcqs generators closed under fiber products), so \cite[Expos\'e VI, Corollaire 5.2]{SGA4} applies.
\end{proof}
\begin{lem}\label{lem: compact Hausdorff abelian is pseudo-coherent}
    Compact Hausdorff abelian groups are pseudo-coherent condensed abelian groups.
\end{lem}
\begin{proof}
    Using the Breen-Deligne resolution (see \cite[Theorem 4.10.]{condensed}) and \cref{lem: Z[S] pseudo-coherent}, we see that compact Hausdorff abelian groups admits a resolution by pseudo-coherent condensed abelian groups.
    The claim follows as in \cite[Tag 064Y]{stacks}.
\end{proof}
\begin{lem}\label{lem: compact Hausdorff exact extension closed subcategory}
    Compact Hausdorff abelian groups (or equivalently condensed abelian groups that are qcqs) are an exact abelian subcategory of condensed abelian groups closed under biproducts and extensions.
\end{lem}
\begin{proof}
    By Pontrjagin duality the category of compact abelian Hausdorff groups \(\mathrm{CompAb}\) is equivalent to the opposite category of abelian groups \(\mathrm{Ab}^{\op}\), so it is abelian.
    Preservation of kernels and products by the inclusion is clear, the only thing left to check is cokernels.
    For this it is sufficient to check that if \(f\from A\to B\) is a surjective morphism of compact Hausdorff abelian groups, then \(\cond{f}\from\cond{A}\to\cond{B}\) is surjective.
    This is clear, as surjective morphisms of compact Hausdorff spaces are a covering on the site of compact Hausdorff spaces, which defines condensed sets and \(\cond{f}\) is just a morphism of representable sheaves.
    For extensions, let \(0\to A\to B\to C\to 0\) be a short exact sequence of condensed abelian groups, such that \(A\) and \(C\) are qcqs.
    We want to show that \(B\) is qcqs.
    Let \(\kappa\) be an uncountable strong limit cardinal, such that \(A\), \(B\) and \(C\) are \(\kappa\)-condensed.
    We have a fiber product diagram 
    \begin{equation*}
        \begin{tikzcd}
            A\times B \arrow[r, "+"] \arrow[d, "\pr"'] & B \arrow[d] \\
            B \arrow[r]                                & C          
            \end{tikzcd}
    \end{equation*}
    where one arrow is the projection and the other arrow is the addition map.
    Clearly \(\pr\from A\times B\to B\) is qcqs (being the base change of \(A\to *\) which is qcqs).
    Since \(B\to C\) is surjective, this implies that \(B\to C\) is qcqs (since \(\kappa\)-condensed sets have a family of quasi-compact generators, see \cite[Expos\'e VI, Proposition 1.10]{SGA4}).
    Since \(C\) is qcqs (and \(*\) is qcqs), it follows that \(B\) is also qcqs.
\end{proof}
\begin{lem}\label{lem: sequential colimit Hausdorff closed under extension finite (co)limits}
    The class of condensed abelian groups that are sequential colimits of compact abelian Hausdorff groups is closed under finite limits, finite colimits, sequential colimits and extensions.
\end{lem}
\begin{proof}
    Compact abelian Hausdorff groups give rise to compact condensed abelian groups by \cref{lem: compact Hausdorff abelian is pseudo-coherent}, so we can get everything by \cref{lem: compact Hausdorff exact extension closed subcategory} except extensions.
    Let \(0\to F\to G\to H\to 0\) be a short exact sequence of condensed abelian groups, where \(F=\colim_{i\in \omega}F_i\), \(H=\colim_{i\in \omega}H_i\) are sequential colimits of compact Hausdorff abelian groups.
    Then \(G=\colim_{i\in\omega}H_i\times_H G\) by universality of colimits in topoi (or rather one should restrict to \(\kappa\)-small things as in the previous proofs).
    We get short exact sequences \(0\to F\to H_i\times_H G\to H_i\to 0\), so we reduce to the case where \(H\) is compact Hausdorff.
    Then \(H\) is pseudo-coherent, so that \(\mathrm{Ext}^1(H,\colim_{i\in \omega}F_i)=\colim_{i\in\omega}\mathrm{Ext}^1(H,F_i)\).
    Let \(\xi\in \mathrm{Ext}^1(H,\colim_{i\in \omega}F_i)\) represent the extension \(0\to F\to G\to H\to 0\).
    Then the previous equiality of Ext-groups shows that there is some index \(j\in\omega\) and a class \(\xi_j\in \mathrm{Ext}^1(H,F_j)\) that maps to \(\xi\).
    From this we obtain \(\xi_i\in \mathrm{Ext}^1(H,F_i)\) for \(i\geq j\) that are just the image of \(\xi_j\) by the map induced by \(F_i\to F_j\), all mapping to \(\xi\).
    This means we find short exact sequences \(0\to F_i\to G_i\to H\to 0\) functorially in \(i\geq j\) fitting into 
    \begin{equation*}
        \begin{tikzcd}
            0 \arrow[r] & F_i \arrow[d] \arrow[r] & G_i \arrow[d] \arrow[r] & H \arrow[d] \arrow[r] & 0 \\
            0 \arrow[r] & F \arrow[r]             & G \arrow[r]             & H \arrow[r]           & 0
            \end{tikzcd}
    \end{equation*} 
    where the left square is a pushout diagram.
    Let us reindex so that \(j=0\).
    Then 
    \begin{equation*}
        \colim_{i\in\omega} G_i=\colim_{i\in\omega} F_i\oplus_{\colim_{i\in\omega}F}\colim_{i\in\omega}G=G
    \end{equation*} 
    where the colimits over \(F\) and \(G\) are as constant diagrams.
    Note that the map \(\colim_{i\in \omega}F_i\to\colim_{i\in\omega}F\cong F\) identifies with the identity map, which gives the second equality.
    Thus we reduce to the case where \(F\) is also compact Hausdorff, where it is \cref{lem: compact Hausdorff exact extension closed subcategory}.
\end{proof}
\begin{lem}\label{lem: baire application}
    Let \(G\) be a compact Hausdorff group such that \(G=\bigcup_{i=0}^{\infty}G_i\) where \(G_i\) are closed subgroups such that \(G_i\subseteq G_{i+1}\).
    Then \(G=G_N\) for some index \(N\).
\end{lem}
\begin{proof}
    By the Baire category theorem there exists some index \(i_0\) such that \(G_{i_0}\) contains an open subset.
    Thus \(G_{i_0}\) is open and closed.
    Applying the Baire category theorem to \(G\cap G_{i_0}^c=\bigcup G_i\cap G_{i_0}^c\) we find another index \(i_1\) such that \(G_{i_1}\) contains an open subset (using that \(G_{i_0}^c\) is open and closed).
    Continuing inductively we find infinitely many indices \(i_j\) such that \(G_{i_j}\) is open and closed and we have \(G=\bigcup_{j=0}^{\infty}G_{i_j}\).
    Since \(G\) is compact, finitely many \(i_j\) suffice and we can choose \(N\) to be the largest one of those finitely many.
    Then \(G=G_N\).
\end{proof}
\begin{lem}\label{lem: sequential colimit qcqs condensed on underlying set}
    Let \(X\) and \(Y\) be condensed abelian groups that are sequential colimits of compact Hausdorff abelian groups.
    Let \(f\from X\to Y\) be a morphism that induces a bijection \(X(*)\to Y(*)\).
    Then \(f\) is already an isomorphism.
\end{lem}
\begin{proof}
    Since \(\ker(f)\) and \(\coker(f)\) are similarly such sequential colimits it suffices to check that if \(X(*)=0\) then \(X=0\).
    Let \(X=\colim_n X_n\) where \(X_n\) come from compact Hausdorff abelian groups.
    Define \(\overline{X}_n\defined\im(X_n\to X)\).
    Then \(X=\colim_n\overline{X}_n\) with \(\overline{X}_n(*)=0\) and \(\overline{X}_n=\colim_m\im(X_n\to X_m)\).
    In this system all transition maps are surjective.
    Thus we may assume that all transition maps \(X_n\to X_{n+1}\) are surjective.
    Then \(X_n=X_0/K_n\) for some closed subgroup \(K_n\) and \(X_0=\bigcup_{n\geq 0}K_n\) and \(K_n\subseteq K_{n+1}\).
    By \cref{lem: baire application} we find that \(X_0=K_n\) for some \(n\) so \(X_n=0\) for some \(X_n\), hence \(X=0\).
\end{proof}
\subsection{Categorical Computations}
\begin{lem}\label{lem: QCoh(Hom(W Gm)) is WMod}
    Let \(W\) be an abelian locally pro-\(p\) group.
    Then \(\qcoh(\Hom(W,\gm))\) is equivalent to the category of smooth \(W\)-representations on \(\zl\)-modules.
\end{lem}
\begin{proof}
    Let \(P\subset P'\) be compact open pro-\(p\) subgroups of \(W\), \(f\from W/P\mhyphen\Mod_{\zl}\to W\mhyphen\Mod_{\zl}\) the functor induced by \(W\to W/P\) and \(f'\) similarly but using \(P'\) instead of \(P\).
    Let \(j_P\from \Hom(W/P,\gm)\to \Hom(W,\gm)\) and \(i\from \Hom(W/P',\gm)\to \Hom(W/P,\gm)\) be the inclusions.
    We get a functor \(\qcoh(\Hom(W,\gm))\to W\mhyphen\Mod_{\zl}\) via \(\mathcal{F}\mapsto\colim_{P'\supset P} f(\Gamma(j_P^*\mathcal{F}))\).
    The transition maps are induced from 
    \begin{equation*}
        i_!j_{P'}^*\mathcal{F}=i_!j_{P'}^*j_P^*\mathcal{F}\to j_P^*\mathcal{F}
    \end{equation*}
    noting that \(\Gamma(i_!j^*_{P'}\mathcal{F})=\Gamma(j_{P'}^*\mathcal{F})\), as \(i\) is the inclusion of an open and closed subscheme.
    It is clear that this will yield a smooth representation, as each \(f(\Gamma(j_P^*\mathcal{F}))\) is.
    We obtain a functor \(W\mhyphen\Mod_{\zl}\to\qcoh(\Hom(W,\gm))\) by mapping \(M\mapsto\colim_{P'\supset P}j_{P'!}\widetilde{M_{P'}}\).
    Note that 
    \begin{equation*}
        j_{P'!}\widetilde{M_{P'}}\subset j_{P!}\widetilde{M_P}
    \end{equation*}
    To see this, observe that \(i^*\widetilde{M_{P}}=\widetilde{M_P}\).
    Thus we get a map
    \begin{equation*}
        i_!\widetilde{M_{P'}}=i_!i^*\widetilde{M_P}\to \widetilde{M_P}
    \end{equation*}
    Now apply \(j_{P!}\) and note that \(j_{P!}i_!i^*=j_{P'!}i^*\).
    One checks that it is an inclusion on stalks.
    For the second functor, note that \(j_{P'}\) is the inclusion of an open and closed subscheme, so it is quasi-coherent.
    Now let us check that the functors are inverse to each other.
    We claim that the map \(\Gamma(j_{P'}^*\mathcal{F})\to\Gamma(j_P^*\mathcal{F})\) is given by 
    \begin{equation*}
        \Gamma(j_{P'}^*)\cong\Gamma(j_P^*\mathcal{F})_{P'/P}\cong\Gamma(j_P^*\mathcal{F})^{P'/P}\subset\Gamma(j_P^*\mathcal{F})
    \end{equation*}
    Let \(t\in\zl[W/P]\) be defined as 
    \begin{equation*}
        t\defined\frac{1}{|P'/P|}\sum_{[g]\in P'/P}g
    \end{equation*}
    Observe that \(t\) is idempotent and that \(D(t)\) precisely cuts out \(\Hom(W/P',\gm)\subset \Hom(W/P,\gm)\).
    The map \(\Gamma(j_{P'}^*\mathcal{F})\to\Gamma(j_P^*\mathcal{F})\) takes an \(s\in\Gamma(j_{P'}^*\mathcal{F})\) and assigns it to the element \(s'\in\Gamma(j_P^*\mathcal{F})\) satisfying \(s'|_{D(t)}=s\) and \(s'|_{D(1-t)}=0\).
    One checks that this shows that the map \(\Gamma(j_{P'}^*\mathcal{F})\to\Gamma(j_P^*\mathcal{F})\) has the form we claimed.
    One way the composite of the two functors sends an \(M\in W\mhyphen\Mod_{\zl}\) to 
    \begin{align*}
        \colim_{P}f(\Gamma(j_P^*(\colim_{P'}j_{P'!}\widetilde{M_{P'}})))&=\colim_{P}f(\Gamma(\colim_{P'}j_P^*(j_{P'!}\widetilde{M_{P'}})))\\
        &=\colim_{P}f(\Gamma(\widetilde{M_P}))\\
        &=\colim_P M_P
    \end{align*}
    Thus we need to show that \(M\cong\colim_P M_P\) naturally in \(M\).
    For this we construct a map \(\colim M_P\to M\).
    Consider an element \([(s,P')]\in\colim M_P\) such that \(s\in M_{P'}\).
    It is identified with \([(s,P)]\) for \(P\subset P'\) via the inclusion 
    \begin{equation*}
        M_{P'}=(M_P)_{P'/P}\cong M_P^{P'/P}\subset M_P
    \end{equation*}
    Taking \(P\) small enough, since \(M\) is a smooth representation, we find a unique lift \(s\in M\).
    It is clear that this map is bijective and the construction is natural in \(M\)
    For the other direction, a sheaf \(\mathcal{F}\) is sent to 
    \begin{equation*}
        \colim_P j_{P!}(\colim_{P'}f'(\Gamma(j_{P'}^*\mathcal{F}))_{P})^{\sim}
    \end{equation*}
    Let us construct isomorphism between \(j_U^*\mathcal{F}\) and \(j_U^*\) applied to this term where \(U\) is some compact open pro-\(p\) subgroup of \(W\).
    \begin{align*}
		j_U^*\colim_{P'} j_{P'!}((\colim_Pf\Gamma(j_P^*\mathcal{F}))_{P'})^\sim&=\colim_{U\supset P'} j_U^*((\colim_Pf\Gamma(j_P^*\mathcal{F}))_{P'})^\sim\\
		&=\colim_{U\supset P'} j_U^*(\colim_{P'\supset P}(\Gamma(j_P^*\mathcal{F})_{P'/P})^\sim\\
		&=\colim_{U\supset P'} j_U^*((\Gamma(j_{P'}^*\mathcal{F}))^\sim\\
		&=j_U^*\mathcal{F}
	\end{align*}
	These isomorphisms glue and are natural in $\mathcal{F}$.
\end{proof}
\begin{cor}\label{cor: subcategory given by e_P}
    We consider the setting of \cref{lem: QCoh(Hom(W Gm)) is WMod}.
    Let \(P\subset W\) be an open pro-\(p\) subgroup.
    By \cref{lem: Hom(T(E) G_m)}, we obtain an idempotent \(e_P\in\oo(\Hom(W,\gm))\), corresponding to the open/closed embedding \(j\from\Hom(W/P,\gm)\injto\Hom(W,\gm)\).
    The subcategory of \(\D(\qcoh(\Hom(W,\gm)))\simeq\D(W\mhyphen\Mod_{\zl})\) spanned by those objects where \(e_P\) acts by the identity is given by \(\D(W/P\mhyphen\Mod_{\zl})\).
\end{cor}
\begin{proof}
    The natural transformation \(e_P\) maybe be described as \(\id\to j_*j^*\cong j_!j^*\to\id\) geometrically, by the proof of \cref{lem: QCoh(Hom(W Gm)) is WMod}.
    In geometry it is clear that the statement holds.
\end{proof}
\begin{rem}
    Using \cref{lem: QCoh(Hom(W Gm)) is WMod} one can easily show that there is an embedding
    \begin{equation*}
        \D(T(E)\mhyphen\Mod_{\zl})\injto\D_{\qcoh}([Z^1(W_E,\widehat{T})/\widehat{T}])
    \end{equation*}
    as predicted by Hellmann in \cite[Conjecture 3.2.]{hellmann}.
\end{rem}
\begin{rem}\label{rem: whittaker sheaf for tori}
	In general the functor in the categorical form of Fargues's conjecture depends on a choice of Whittaker datum, determining the image of the structure sheaf of $[Z^1(W_E,\widehat{T})/\widehat{T}]$.
	For $T$ the unipotent radical is trivial, so there is a unique choice of Whittaker datum, giving us the representation $\cind_{*}^{T(E)}1$
	On the other hand, the $T(E)$-representation corresponding to the structure sheaf on $[Z^1(W_E,\widehat{T})/\widehat{T}]$ under the equivalence \(\qcoh([Z^1(W_E,\widehat{T})/\widehat{T}])_0\simeq\qcoh(\Hom(T(E),\gm))\simeq T(E)\mhyphen\Mod_{\zl}\) is given by the colimit $\varinjlim\zl[T(E)/U]$ over some neighborhood basis of the identity \(U\)
	\begin{align*}
        \zl[T(E)/U]&=\zl[T(E)/U']_{U}\\
        &\cong\zl[T(E)/U']^{U/U'}\\
        &\subset\zl[T(E)/U']
    \end{align*}
    whenever \(U'\subset U\).
	Each $\zl[T(E)/U]$ can be considered as functions $T(E)\to\zl$ with support being the union of finitely many translates of $T(E)/U$.
	We therefore get a map
	\[
		\bigcup_{U\subset T(E)}\zl[T(E)/U]\to\cind_{*}^{T(E)}1
	\]
	This is compatible with the transition maps and since $\bigcap U=*$, it is an isomorphism.
    Using \cref{cor: subcategory given by e_P} we can compute the coinvariants as \((\cind_{*}^{T(E)}1)_{U}=\zl[T(E)/U]\) when \(U\subset T(E)\) is an open pro-\(p\) subgroup.
\end{rem}
\subsection{Sheaves on gerbes banded by diagonalizable groups}\label{sheaves on gerbes banded by diagonalizable groups}
Let \(\mathcal{X}\to X\) a gerbe over an algebraic space, banded by some diagonalizable group \(\Delta\).
In this case the inertia stack is given by \(\Delta\times \mathcal{X}\) and thus we get an action of \(\Delta\) on any sheaf on \(\mathcal{X}\).
If \(\chi\) is some character of \(\Delta\), let \(\D_{\qcoh}(\mathcal{X})_{\chi}\) denote the full subcategory of \(\D_{\qcoh}(\mathcal{X})\) where the action of \(\Delta\) on the cohomology sheaves is given by \(\chi\).
Then we have the following result for the derived category of quasi-coherent sheaves on \(\mathcal{X}\):
\begin{prop}\label{prop: quasi-coherent sheaves on gerbes}
    We have a product decomposition 
    \begin{equation*}
        \D_{\qcoh}(\mathcal{X})\simeq\prod_{\chi\in X^*(\Delta)}\D_{\qcoh}(\mathcal{X})_{\chi}
    \end{equation*}
    Additionally we have an equivalence of stable \(\infty\)-categories 
    \begin{equation*}
        \pi^*\from\D_{\qcoh}(X)\simeq\D_{\qcoh}(\mathcal{X})_0
    \end{equation*}
\end{prop}
\begin{proof}
    This is \cite[Theorem 5.4.]{sheaves_on_gerbes} and \cite[Proposition 5.7.]{sheaves_on_gerbes}.
\end{proof}
From this we can easily deduce the corresponding result for perfect complexes, where we write \(\perf(\mathcal{X})_{\chi}\) for \(\perf(X)\cap\D_{\qcoh}(\mathcal{X})_\chi\):
\begin{lem}\label{lem: perfect complexes on gerbes}
    We have a direct sum decomposition 
    \begin{equation*}
        \perf(\mathcal{X})\simeq\bigoplus_{\chi\in X^*(\Delta)}\perf(\mathcal{X})_{\chi}
    \end{equation*}
    Additionally we have an equivalence of stable \(\infty\)-categories 
    \begin{equation*}
        \pi^*\from\perf(X)\simeq\perf(\mathcal{X})_0
    \end{equation*}
\end{lem}
\begin{proof}
    Note that the direct product decomposition for the derived category of quasi-coherent sheaves is realized by writing \(\mathcal{F}=\bigoplus_{\chi\in X^*(\Delta)}\mathcal{F}_\chi\) where each \(\mathcal{F}_\chi\) is the \(\chi\)-eigenspace for the action of \(\Delta\).
    Thus if \(\mathcal{F}\) is a perfect complex then \(\mathcal{F}_{\chi}\) are perfect complexes too.
    Conversely a finite direct sum of perfect complexes is perfect.
    This gives the first claim.
    For the second claim we want to check that \(\pi^*\) restricts to an equivalence on perfect objects.
    Note the quasi-inverse for \(\pi^*\) is given by \(\pi_*\), so it suffices to check that both of these functors preserve perfect objects. 
    This is clear for \(\pi^*\).
    For \(\pi_*\) we have base change as \(\pi\) is qcqs, so we may assume that we are working with \(B\Delta\to\spec(A)\) for some ring \(A\).
    In this case the claim is clear.
\end{proof}
\begin{cor}
    We have a canonical isomorphism \(\oo([Z^1(W_E,\widehat{T})/\widehat{T}])\cong\varprojlim_{K\subset T(E)}\zl[T(E)/K]\).
\end{cor}
\begin{proof}
    This is immediate from \cref{lem: perfect complexes on gerbes} and \cref{lem: geometric incarnation langlands duality tori}.
\end{proof}
\begin{rem}\label{rem: our iso Z^spec(T) is the Fargues-Scholze one}
    This was already obtained as \cite[Proposition IX.6.4.]{geometrization}, where they prove it using resolution of tori by induced tori to reduce it to \(\gm\) where it follows from local class field theory.
    Our isomorphism agrees with theirs, as our isomorphism is natural in \(T\) and for \(T=\gm\) it agrees with theirs.
\end{rem}
\subsection{Some general topology}\label{some general topology}
\begin{lem}\label{lem: neighborhood basis construction}
    Let \(G\) be a locally profinite group.
    Let \(P_i\), \(i\in I\) be a system of open subgroups such that \(\bigcap_{i\in I}P_i=\{1\}\), such that for any pair \(i,i'\in I\) there is a \(j\in I\) satisfying \(P_j\subset P_i\cap P_{i'}\) and there is a compact open subgroup \(H\subset G\) such that \(P_i\subset H\) for all \(i\).
    Then the \(P_i\) are a neighborhood basis of the identity.
\end{lem}
\begin{proof}
    By replacing \(G\) with \(H\) we may assume that \(G\) is profinite, in particular compact.
    We know that compact open subgroups form a neighborhood basis, let \(U\) be such a subgroup.
    We observe for later that all the \(P_i\) are also closed, as they are open subgroups.
    Assume none of the \(P_i\) are contained in \(U\).
    This means that for all \(i\in I\) we \(U^c\cap P_i\neq\emptyset\) for the complement \(U^c\) of \(U\).
    By assumption we have \(U^c\cap\bigcap_{i\in I}P_i=\emptyset\).
    This is an infinite intersubsection of closed subsets in a compact space.
    This shows that there must be finitely many \(i_1,\dots,i_n\) such that \(U^c\cap P_{i_1}\cap\dots\cap P_{i_n}=\emptyset\).
    By assumption there is a \(P_i\subset P_{i_1}\cap\dots\cap P_{i_n}\), so then \(P_i\cap U^c=\emptyset\), a contradiction.
\end{proof}
\begin{lem}\label{lem: condensed H_1 of quotient is torus quotient}
    Let \(P\) be an open subgroup of \(F^\times\).
    Define a condensed subgroup 
    \begin{equation*}
        \Theta(P)\defined\ker(H_1(W_{F/E},\widehat{L})\to H_1(W_{F/E}/P,\widehat{L})))
    \end{equation*}
    of \(T(E)\).
    Then \(\Theta(P)\) comes from a topological space, is an open subgroup of \(T(E)\) and the \(\Theta(P)\) form a neighborhood basis of the identity if we let \(P\) vary.
\end{lem}
\begin{proof}
    First of all note that the kernel of a morphism of condensed groups coming from compactly generated Hausdorff groups is itself a topological group, as the inclusion of compactly generated Hausdorff spaces to condensed sets is limit preserving and fully faithful.
    Since \(W_{F/E}/P\) is discrete, \(H_1(W_{F/E}/P,\widehat{L})\) is a discrete set.
    This shows that \(\Theta(P)\) is open.
    Consider the following diagram:
    \begin{equation*}
        \begin{tikzcd}
            {H_1(W_{F/E},\widehat{L})} \arrow[r, "f"] \arrow[rd, "h"'] & {H_1(W_{F/E}/P_1,\widehat{L})\oplus H_1(W_{F/E}/P_2,\widehat{L})} \\
                                                                       & {H_1(W_{F/E}/(P_1\cap P_2),\widehat{L})} \arrow[u, "g"']         
            \end{tikzcd}
    \end{equation*}
    Then \(\Theta(P_1)\cap \Theta(P_2)=\ker(f)\supset\ker(h)=\Theta(P_1\cap P_2)\).
    Let \(P_i\subset F^\times\) be a collection of Galois-stable open subgroups such that \(\bigcap_{i\in I}P_i=\{1\}\).
    One example would be to take the higher unit groups \(P_i=U_F^{(i)}\).
    By the (non-condensed) inflation-restriction sequence and \cref{rem: underlying set condensed homology is ordinary homology} all elements of \(\Theta(P)(*)\) have a representative of the form \(\sum m_i\otimes [g_i]\) with \(g_i\in P\).
    Thus \(\Theta(P)(*)\) lies in the image of the morphism of non-condensed homology groups \(f\from H_1(P,\widehat{L})\subset H_1(F^\times,\widehat{L})\to H_1(W_{F/E},\widehat{L})\) (or equivalently the map on underlying sets of the condensed homology groups).
    Under the identifications with \(T(P)\), \(T(F)\) and \(T(E)\) we see by using \cref{lem: map on homology is norm} that the image of \(f\) is given by \(\nm(T(P))\subset T(E)\subset T(F)\).
    Note that \(\nm(T(P))\subset T(P)\) by Galois-stability of \(P\), so we have \(\bigcap T(P_i)\subset \bigcap\nm T(P_i)\subset \bigcap P_i=\{1\}\).
    When \(P'\subset P\) we have maps 
    \begin{equation*}
        H_1(W_{F/E},\widehat{L})\to H_1(W_{F/E}/P')\to H_1(W_{F/E}/P,\widehat{L})
    \end{equation*}
    thus \(\Theta(P')\subset \Theta(P)\).
    By the condensed inflation-restriction we have a surjection \(H_1(P,\widehat{L})_{W_{F/E}/P}\to\Theta(P)\).
    Thus we get a surjection \(P^{\mathrm{rk}(T)}\cong H_1(P,\widehat{L})\to\Theta(P)\).
    We know that \(\Theta(P)\) comes from a compactly generated Hausdorff space, so \(P^{\mathrm{rk}(T)}\to\Theta(P)\) comes from a surjective map of topological spaces.
    Thus if \(P\) is compact \(\Theta(P)\) comes from a compact space.
    Therefore the \(\Theta(P)\) form a neighborhood basis of the identity by \cref{lem: neighborhood basis construction}.
\end{proof}
\begin{rem}
    It seems reasonable to conjecture \(\Theta(P)=\nm(T(P))\), but we will not pursue this.
\end{rem}
\begin{lem}\label{lem: map on homology is norm}
    Let \(P\) be an open subgroup of \(F^\times\) or the trivial subgroup.
    Let \(E'\) be the fixed field for some subgroup of \(Q\).
    Under the isomorphisms \(H_1(W_{F/E}/P,\widehat{L})\cong T(E)/\Theta(P)\) and \(H_1(W_{F/E'}/P,\widehat{L})\cong T(E')/\Theta(P)\) the map \(H_1(W_{E'}/W_P,\widehat{L})\to H_1(W_E/W_P,\widehat{L})\) corresponds to the norm map \(T(E')\to T(E)\).
\end{lem}
\begin{proof}
    This is a special case of \cref{lem: homology inclusion is norm map}. 
\end{proof}
	\section{Compatibilities of the spectral action}\label{spectral action}
Let \(G\) be a reductive group over \(E\), \(\Lambda\) a discrete valuation ring over \(\zl[\sqrt{q}]\), \(\ell\) a prime that does not divide \(|\pi_0(Z(G))|\) or invertible in \(\Lambda\) and \(Q\) a finite quotient of \(W_E\) over with the action of \(W_E\) on \(\widehat{G}\) factors.
In \cite[section X.3.]{geometrization}, Fargues and Scholze construct an action of \(\perf([Z^1(W_E,\widehat{G})_{\Lambda}/\widehat{G}])\) on \(\D_{\lis}(\bun{G},\Lambda)^{\omega}\).
This gives us a ring map \(\bcenter^{\mathrm{spec}}(G,\Lambda)\to\bcenter^{\mathrm{geom}}(G,\Lambda)\) by the action of \(\bcenter^{\mathrm{spec}}(G,\Lambda)\) on \(\oo *-\).
Now fix an open subgroup \(P\) of the wild inertia and a discretization \(W\subset W_E/P\).
By \cite[subsection VIII.4.]{geometrization} we also get a map \(\exc(W,G)_{\Lambda}\to\bcenter^{\mathrm{geom}}(G,\Lambda)\).
Since both of these constructions come from Hecke operators one would expect that they are compatible with each other.
We prove that this is the case.

\subsection{Unwinding identifications}
Fix \(R\) a discrete valuation ring, \(H\) an affine algebraic group over \(R\), \(\Gamma\) a finitely generated group, \(\Theta\) a finite quotient of \(\Gamma\) acting on \(H\).
The goal of this subsection is to understand the \(\Gamma\)-action on \(\ev^*\mathcal{F}\) where \(\ev\from[Z^1(\Gamma,H)/H]\to B(H\rtimes \Theta)\) is the evaluation map and \(\mathcal{F}\in\qcoh(B(H\rtimes\Theta))\).
\begin{lem}\label{lem: 2-morphism on pullbacks}
    We use the notation of \cite[Tag 0440]{stacks} in this lemma.

    Let \(f,g\from (U,R,s,t,c)\to (U',R',s',t',c')\) be two morphisms between groupoids in algebraic spaces, let \(\eta\from U\to R'\) define a natural transformation.
    Let \((\mathcal{F},\alpha)\) be a quasi-coherent module on \((U',R',s',t',c')\).
    The natural transformation \(\eta\) defines a 2-morphism between \(f\from [U/R]\to [U'/R']\) to \(g\from [U/R]\to[U'/R']\).
    Thus it introduces a map \(\eta^*\from g^*\mathcal{F}\to f^*\mathcal{F}\) (identifying quasi-coherent sheaves on the stack with quasi-coherent sheaves on presentations via \cite[Tag 06WT]{stacks}).
    It is given by \(g^*\mathcal{F}=\eta^*{t'}^*\mathcal{F}\xto{\eta^*\alpha}\eta^*{s'}^*\mathcal{F}=f^*\mathcal{F}\).
\end{lem}
\begin{proof}
    Note that since \(\eta\) is a natural transformation we have \(s'\circ\eta=f\) and \(t'\circ\eta=g\).
    The lemma follows from unwinding the proof of \cite[Tag 06WT]{stacks}.
\end{proof}
    We have an isomorphism \(\maps_{B\Theta}(B\Gamma,B(H\rtimes\Theta))=[\spec(A)/H\rtimes \Theta]\), where \(\spec(A)\) is the closed subscheme of \(\Hom(\Gamma,H\rtimes \Theta)\times \Theta\) classifying group homomorphisms \(f\from \Gamma\to H(\Lambda)\rtimes \Theta(\Lambda)\) and \(q\in \Theta(\Lambda)\) such that postcomposing \(f\) with the projection \(H\rtimes \Theta\to \Theta\) and conjugating with \(q\) gives the canonical map \(\Gamma\to \Theta\).
    Evaluation at an element \(n\in\Gamma\) gives a map \(\ev_n\from\spec(A)\to H\rtimes \Theta\), inducing a map on coordinate rings which we will denote \(h\mapsto n^*h\).
    There is a canonical evaluation map \(f\from\maps_{B\Theta}(B\Gamma,B(H\rtimes\Theta))\to B(H\rtimes\Theta)\).
\begin{lem}\label{lem: spectral action spectral side general nonsense}
    Let \(V\) be a \(\oo(H\rtimes\Theta)\)-comodule, which we consider as a quasi-coherent sheaf on the classifying stack \(B(H\rtimes\Theta)\).
    Then the pullback naturally carries an action of \(\Gamma\) as described in \cite[Theorem X.1.1.]{geometrization}.

    If the comodule structure on \(V\) is given by \(v\mapsto\sum v_i\otimes h_i\), then the action of an element \(n\in\Gamma\) on the pullback sheaf \(V\otimes A\) is given by \(n.(v\otimes 1)=\sum v_i\otimes n^*h_i\).
    Identifying \([\spec(A)/H\rtimes\Theta]=[Z^1(\Gamma,H)/H]\), the corresponding action on \(V\otimes\oo(Z^1(\Gamma,H))\) is given by \(n.(v\otimes 1)=\sum v_i\otimes n^*h_i\) too, where \(h\mapsto n^*h\) corresponds to the natural evaluation map \(\ev_n\from \Hom(\Gamma,H)\to H\to H\rtimes\Theta\)
\end{lem}
\begin{proof}
    Each element \(n\in\Gamma\) gives 2-endomorphism of the natural map 
    \begin{equation*}
        \maps_{B\Theta}(B\Gamma,B(H\rtimes\Theta))\to B(H\rtimes\Theta)
    \end{equation*}
    that is given by \(\ev_n\from\spec(A)\to H\rtimes\Theta\).
    The first part of the claim follows by unwinding what \cref{lem: 2-morphism on pullbacks} means when everything is affine.
    For the second part, note that the map \([Z^1(\Gamma,H)/H]\xto{\cong}[\spec(A)/H\rtimes\Theta]\) induced by the natural inclusion \(Z^1(\Gamma,H)\to\spec(A)\) is an isomorphism.
    Now one just needs to unwind the identifications.
\end{proof}
\begin{lem}
    Keep the notation of \cref{lem: spectral action spectral side general nonsense}.
    Let \(S\) be a finite generating set for \(\Gamma\).
    Then we can write \(\Hom(\Gamma,H)=\spec(\oo(H)^{\otimes S}/I)\) for some ideal \(I\).
    Then \(\ev_s\) for \(s\in S\) is given by \(h\mapsto 1\otimes\dots\otimes h\otimes\dots\otimes 1\) where the \(h\) is in the \(s\)'th position.
    For general \(n=s_1\cdot\dots\cdot s_l\in\Gamma\) with \(s_i\in\Gamma\) we have \(\ev_n(h)=\ev_{s_1}(h)\dots\ev_{s_l}(h)\).
\end{lem}
\begin{proof}
    \(\spec(\oo(H)^{\otimes S}/I)\) represents the functor sending a ring \(\Lambda\) to the set of set-theoretic maps \(S\to H(\Lambda)\).
    Saying that this defines a group homomorphism \(\Gamma\to H(\Lambda)\) is a closed subscheme cut out by some ideal \(I\).
    The rest is obvious.
\end{proof}
\subsection{Compatibility with the excursion algebra}\label{compatiblity with the excursion algebra}
Let \(\Lambda\) be a discrete valuation ring.
For this subsection, let everything be base changed to \(\Lambda\), that is we write \(\exc(W,\widehat{G})\) for \(\exc(W,\widehat{G})_{\Lambda}\) and \(Z^1(W,\widehat{G})\) for \(Z^1(W,\widehat{G})_{\Lambda}\).
Let \(\mathcal{C}\) be an idempotent-complete small stable \(\Lambda\)-linear \(\infty\)-category.
Assume we have functorially in finite sets \(I\) exact \(\rep_{\Lambda}(Q^I)\)-linear functors 
\begin{equation*}
    \rep_{\Lambda}((\widehat{G}\rtimes Q)^I)\to\End_{\Lambda}(\mathcal{C})^{BW^I}
\end{equation*}
where \(W\) is a discretization of \(W_E/P\) for \(P\) an open subgroup of the wild inertia.
\begin{thm}\label{lem: excursion algebra action compatible with spectral action}
    We obtain a map \(\exc(W,\widehat{G})\to\pi_0\bcenter(\mathcal{C})\) from the excursion algebra for \(\widehat{G}\) to the Bernstein center of \(\mathcal{C}\).
    The spectral action gives us a map \(\oo(Z^1(W,\widehat{G}))^{\widehat{G}}\to\pi_0\bcenter(\mathcal{C})\).
    The following diagram commutes 
    \begin{equation*}
        \begin{tikzcd}
            {\exc(W,\widehat{G})} \arrow[d] \arrow[r]         & \pi_0\bcenter(\mathcal{C}) \\
            {\oo(Z^1(W,\widehat{G}))^{\widehat{G}}} \arrow[ru] &                           
            \end{tikzcd}
    \end{equation*}
\end{thm}
\begin{proof}
    By abuse of notation we will consider \(\oo_{F_n}\defined\oo_{[Z^1(F_n,\widehat{G}/\widehat{G}]}\) and \(\oo_{W}\defined\oo_{[Z^1(W,\widehat{G})/\widehat{G}]}\) as objects of \(\perf(\maps_{BQ}^\Sigma(BF_n,B(\widehat{G}\rtimes Q)))\) and \(\perf(\maps_{BQ}^\Sigma(BW,B(\widehat{G}\rtimes Q)))\) respectively.
    This is possible by \cite[Proposition X.3.3.]{geometrization} and (the consequence of) \cite[Proposition X.3.4.]{geometrization}.
    It suffices to check that 
    \begin{equation}\label{eq: diagram that ought to commute spectral action excursion algebra}
        \begin{tikzcd}
            {\oo(Z^1(F_n,\widehat{G}))^{\widehat{G}}} \arrow[d] \arrow[r]         & \pi_0\bcenter(\mathcal{C}) \\
            {\oo(Z^1(W,\widehat{G}))^{\widehat{G}}} \arrow[ru] &                           
            \end{tikzcd}
    \end{equation}
    commutes for every \(F_n\to W\), where the map \(\oo(Z^1(F_n,\widehat{G}))^{\widehat{G}}\to\pi_0\bcenter(\mathcal{C})\) is constructed as in the proof of \cite[Theorem VIII.4.1.]{geometrization}.

    The action of \(\perf(\maps_{BQ}^\Sigma(BW,B(\widehat{G}\rtimes Q)))\simeq \perf(\maps_{BQ}(BW,B(\widehat{G}\rtimes Q)))\) is constructed by taking a sifted colimit of actions by \(\perf(\maps_{BQ}^\Sigma(BF_n,B(\widehat{G}\rtimes Q)))\) and sifted colimits of monoidal stable \(\infty\)-categories are computed on the underlying \(\infty\)-categories.
    Thus via the functor 
    \begin{equation*}
        \perf(\maps_{BQ}^\Sigma(BF_n,B(\widehat{G}\rtimes Q)))\to\perf(\maps_{BQ}^\Sigma(BW,B(\widehat{G}\rtimes Q)))
    \end{equation*}
    the spectral actions gives us a commutative diagram 
    \begin{equation*}
        \begin{tikzcd}
            {\oo(Z^1(F_n,\widehat{G}))^{\widehat{G}}} \arrow[d] \arrow[r] & \pi_0\bcenter(\mathcal{C}) \\
            {\oo(Z^1(W,\widehat{G}))^{\widehat{G}}} \arrow[ru]            &                           
            \end{tikzcd}
    \end{equation*}
    via the actions of \(\oo(Z^1(F_n,\widehat{G}))^{\widehat{G}}\) on \(\oo_{F_n}*-\) and \(\oo(Z^1(W,\widehat{G}))^{\widehat{G}}\) on \(\oo_{W}*-\).
    The diagram fits into \eqref{eq: diagram that ought to commute spectral action excursion algebra} above by being part of the diagram 
    \begin{equation*}
        \begin{tikzcd}
            {\oo(Z^1(F_n,\widehat{G}))^{\widehat{G}}} \arrow[r, "\eta"] \arrow[d,"\id"]    & \pi_0\bcenter(\mathcal{C}) \\
            {\oo(Z^1(F_n,\widehat{G}))^{\widehat{G}}} \arrow[d] \arrow[ru, "\alpha"] &                            \\
            {\oo(Z^1(W,\widehat{G}))^{\widehat{G}}} \arrow[ruu]                      &                           
            \end{tikzcd}
    \end{equation*}
    The outer triangle is the diagram \eqref{eq: diagram that ought to commute spectral action excursion algebra}.
    Thus to show that the outer triangle commutes it suffices to check that the upper triangle commutes, where \(\eta\) comes from excursion data and \(\alpha\) comes from the spectral action.
    Thus to prove the theorem we may replace \(W\) by a free group \(F_n\) on \(n\) generators with a map \(F_n\to Q\).
    Recall that given an excursion datum \((I,V,\alpha,\beta,(\gamma_i)_{i\in I})\) with \(I\) a finite set, \(V\in\rep_{\Lambda}(\widehat{G}\rtimes Q)\), \(\alpha\from 1\to V|_{\rep_{\Lambda}(\widehat{G})}\), \(\beta\from V|_{\rep_{\Lambda}(\widehat{G})}\to 1\) and \(\gamma_i\in F_n\) we  get an element in \(\pi_0\bcenter(\mathcal{C})\) by considering 
    \begin{equation*}
        \id=T_1\xto{T_{\alpha}}T_V\xto{(\gamma_i)_{i\in I}}T_V\xto{T_\beta} T_1=\id
    \end{equation*}
    In general for any anima \(S\) over \(BQ\) we have the following natural commuting triangle 
    \begin{equation*}
        \begin{tikzcd}
            {\perf(\maps_{BQ}^\Sigma(S,B(\widehat{G}\rtimes Q)))^S} \arrow[r] & {\perf(\maps_{BQ}(S,B(\widehat{G}\rtimes Q)))^S}                                                                \\
                                                                                     & \rep_{\Lambda}(\widehat{G}\rtimes Q)\subset\perf(B(\widehat{G}\rtimes Q)) \arrow[lu, "\ev_{\Sigma}^*"] \arrow[u, "\ev^*"']
            \end{tikzcd}
    \end{equation*}
    This is because the map \(\ev_\Sigma^*\) is defined via sifted colimits from the case where \(S\) is a finite set.
    For finite sets this diagram trivially exists.
    In particular for \(S=BF_n\) we have \(\ev^*V\in\perf(\maps_{BQ}^\Sigma(BF_n,B(\widehat{G}\rtimes Q)))\) for any \(V\in\rep_{\Lambda}(\widehat{G}\rtimes Q)\) and \(\ev^*V=\ev^*_{\Sigma}V\).
    This means given an excursion datum as above we may construct the element in \(\pi_0\bcenter(\mathcal{C})\) by considering 
    \begin{equation*}
        \id=\oo_{F_n}*-=\ev^*1*-\xto{\ev^*\alpha}\ev^*V*-\xto{(\gamma_i)_{i\in I}}\ev^*V*-\xto{\ev^*\beta}\ev^*1*-=\oo_{F_n}*-=\id
    \end{equation*}
    From this we see that the induced element in the center \(\pi_0\bcenter(\mathcal{C})\) comes from an endomorphism of \(\oo_{F_n}\).
    Thus we may assume that \(\mathcal{C}\) is \(\perf(\maps_{BQ}(BF_n,B(\widehat{G}\rtimes Q))\).
    Unwinding the construction of the map \(\exc(F_n,\widehat{G})\to\pi_0\bcenter(\mathcal{C})\) we must show the following:

    Given 
    \begin{equation*}
        f\in\oo(\widehat{G}\backslash (\widehat{G}\rtimes Q)^{\{0,\dots,n\}}/\widehat{G})\subset\oo((\widehat{G}\rtimes Q))^{\otimes\{0,\dots,n\}}
    \end{equation*}
    that can be written as \(f=\sum_{i\in I}f_0^{(i)}\otimes\dots\otimes f_n^{(i)}\),
    we construct \(V_f\) as the \((\widehat{G}\rtimes Q)^{\{0,\dots,n\}}\)-subrepresentation of \(\oo((\widehat{G}\rtimes Q)^{\{0,\dots,n\}}/\widehat{G})\) spanned by \(f\).
    It is equipped with maps \(\alpha_f\from 1\to V_f\) induced by \(f\) and \(\beta_f\from V_f\to 1\) induced by evaluation at \(1\in(\widehat{G}\rtimes Q)^{\{0,\dots,n\}}\).
    The map \(\oo_{F_n}\to\oo_{F_n}\) induced by the excursion datum \((V_f,\alpha_f,\beta_f,(1,\gamma_1,\dots,\gamma_n))\) where \(\gamma_i\) are the generators of \(F_n\) is given by 
    \begin{equation*}
        \sum_{i\in I}\varepsilon(f_0^{(i)})\otimes\overline{f_1^{(i)}}\otimes\dots\otimes\overline{f_n^{(i)}}
    \end{equation*}
    where \(\overline{\varepsilon}\from\oo(\widehat{G}\rtimes Q)\to\Lambda\) is the counit and \(f_i\mapsto\overline{f_i}\) is the natural map \(\oo(\widehat{G}\rtimes Q)\to\oo(\widehat{G})\). 
    Since everything in the following will be linear, we assume that \(|I|=1\) for notational simplicity and drop the superscripts.

    Observe that \(V_f\subset\oo((\widehat{G}\rtimes Q)^{\{0,\dots,n\}}/\widehat{G})\subset\oo((\widehat{G}\rtimes Q)^{\{0,\dots,n\}})\) as \((\widehat{G}\rtimes Q)^{\{0,\dots,n\}}\)-representations.
    This means the \(\oo((\widehat{G}\rtimes Q)^{\{0,\dots,n\}})\)-comodule structure is easily described.
    Namely let \(g=g_0\otimes\dots\otimes g_n\in V_f\). The coaction \(\rho\) agrees with the comultiplication \(\Delta\) for \(\oo((\widehat{G}\rtimes Q)^{\{0,\dots,n\}})\).
    If \(\overline{\Delta}\) denotes the comultiplication on \(\oo(\widehat{G}\rtimes Q)\), then 
    \begin{equation*}
        \rho(g)=\Delta(g)=\overline{\Delta}(g_0)\otimes\dots\otimes\overline{\Delta}(g_n)
    \end{equation*}
    Let \(h\in\oo(\widehat{G}\rtimes Q)\).
    If we write \(\oo(Z^1(F_n,\widehat{G}))=\oo(\widehat{G})^{\otimes n}\), then \(\gamma_i^*h=1\otimes\dots\otimes \overline{h}\otimes\dots\otimes 1\) where \(\overline{h}\) is in the \(i\)-th position and \(1^*h=\varepsilon(h)\).
    Here we use the notation of \cref{lem: spectral action spectral side general nonsense}.
    Thus the action of \((1,\gamma_1,\dots,\gamma_n)\) on \(f\in\ev^*V_f=V_f\otimes\oo(Z^1(F_n,\widehat{G}))\) is given by 
    \begin{equation*}
        (1,\gamma_1,\dots,\gamma_n).f=(\id\otimes 1^*)(\overline{\Delta}(f_0))\otimes (\id\otimes \gamma_1^*)(\overline{\Delta}(f_1))\otimes\dots\otimes (\id\otimes \gamma_n^*)(\overline{\Delta}(f_n))
    \end{equation*}
    finally, applying \(\ev^*V_f\to\ev^*1\) is just given by applying \(\varepsilon\otimes \id\).
    If we let \(\overline{\varepsilon}\) denote the counit of \(\oo((\widehat{G}\rtimes Q)^{\{0,\dots,n\}})\), then \(\varepsilon=\overline{\varepsilon}^{\otimes n}\)
    Note that the following diagram commutes by definition 
    \begin{equation*}
        \begin{tikzcd}
            \oo(\widehat{G}\rtimes Q) \arrow[r, "\overline{\Delta}"] \arrow[rd, "\id"'] & \oo(\widehat{G}\rtimes Q)\otimes \oo(\widehat{G}\rtimes Q) \arrow[d, "\overline{\varepsilon}\otimes\id"] \arrow[r] & \oo(\widehat{G}\rtimes Q)\otimes\oo(\widehat{G}) \arrow[ldd, "\overline{\varepsilon}\otimes\id"] \\
                                                                                        & \oo(\widehat{G}\rtimes Q) \arrow[d]                                                                                &                                                                                                  \\
                                                                                        & \oo(\widehat{G})                                                                                                    &                                                                                                 
            \end{tikzcd}
    \end{equation*}
    and \((\overline{\varepsilon}\otimes\overline{\varepsilon})\circ\overline{\Delta}\cong\overline{\varepsilon}\).
    Thus we have 
    \begin{equation*}
        (\varepsilon\otimes\id)((\id\otimes 1^*)(\overline{\Delta}(f_0))\otimes (\id\otimes \gamma_1^*)(\overline{\Delta}(f_1))\otimes\dots\otimes (\id\otimes \gamma_n^*)(\overline{\Delta}(f_n)))=\varepsilon(f_0)\otimes\overline{f_1}\otimes\dots\otimes\overline{f_n}
    \end{equation*}
    like we wanted.
\end{proof}
\subsection{Compatibility with the \texorpdfstring{$\pi_1(G)_Q$}{pi1(G)Gamma}-grading}\label{compatibility spectral action with grading}
Let \(\Lambda\) be a \(\zl[\sqrt{q}]\)-algebra.
Given a representation \(V\in\rep_{\Lambda}(\widehat{G}\rtimes W_E)\) we can decompose it into a direct sum \(V=\bigoplus_{\chi\in X^*(Z(\widehat{G})^Q)}V_{\chi}\), where \(V_{\chi}\) is the subrepresentation of \(V\) where \(Z(\widehat{G})^Q\) acts via \(\chi\).
This gives rise to a decomposition \(\rep_{\Lambda}(\widehat{G}\rtimes W_E)=\bigoplus_{\chi\in X^*(Z(\widehat{G})^Q)}\rep_{\Lambda}(\widehat{G}\rtimes W_E)_{\chi}\).
On the other hand \cite[Proposition III.3.6.]{geometrization} gives us a canonical map \(\pi_0(\gr_{G,\Div^1_X})\to\Gamma\backslash\pi_1(G)\).
Postcomposing this map with \(\Gamma\backslash\pi_1(G)\to\pi_1(G)_Q\cong X^*(Z(\widehat{G})^Q)\) this gives rise to a decomposition 
\begin{equation*}
    \mathrm{Sat}_G(\Lambda)=\bigoplus_{\chi\in X^*(Z(\widehat{G})^Q)}\mathrm{Sat}_{G,\chi}(\Lambda)
\end{equation*}
since \(\mathrm{Sat}_G(\zl[\sqrt{q}])\subset\D_{\et}(\gr_{G,\Div^1_X},\zl[\sqrt{q}])\) by \cite[Proposition VI.7.2.]{geometrization} (and then extending to \(\Lambda\)-valued coefficients by extending linearly).
\begin{lem}\label{lem: geometric satake decomposition compatible}
    Under the geometric Satake equivalence these two decompositions match up.
\end{lem}
\begin{proof}
    By Galois-decent we may assume that \(G\) is split.
    Additionally we may assume \(\Lambda=\zl[\sqrt{q}]\).
    Note that for the sheaf \(\mathcal{S}_V\) attached to some representation \(V\) has no \(\ell\)-torsion, since the representations are on finite projective \(\zl[\sqrt{q}]\)-modules.
    Therefore we can invert \(\ell\) and work with \(\mathrm{Sat}_G(\bbQ_{\ell}[\sqrt{q}])\), as the support will not change.
    As in the proof of \cite[Theorem VI.11.1]{geometrization} we write 
    \begin{equation*}
        \mathrm{Sat}_G(\bbQ_{\ell}[\sqrt{q}])=\bigoplus_{\mu}\mathrm{Rep}_{W_E}^{\mathrm{cont}}(\bbQ_{\ell}[\sqrt{q}])\otimes A_\mu
    \end{equation*}
    with \(A_\mu=j_{\mu!*}\bbQ_{\ell}[\sqrt{q}][d_\mu]\) for \(\mu\) running over the dominant cocharacters and \(j_\mu\) the inclusion of the open Schubert cell attached to \(\mu\).
    \(\rep_{\bbQ_{\ell}[\sqrt{q}]}(\widehat{G}\rtimes W_E)\) admits a similar description where we replace \(A_\mu\) with the irreducible representation \(L(\mu)\) attached to \(\mu\).
    The geometric Satake equivalence maps \(A_\mu\) to \(L(\mu)\).
    Using the description of certain open/closed subspaces of \(\gr_G\) indexed by \(\pi_1(G)\) as in \cite[Proposition III.3.6.(iii)]{geometrization} finishes the proof.
\end{proof}
Recall that we have \(\pi_0|\bun{G}|=\pi_1(G)_Q\) canonically.
This gives a direct sum decomposition 
\begin{equation*}
    \D_{\lis}(\bun{G},\Lambda)^{\omega}=\bigoplus_{\alpha\in\pi_1(G)_Q}\D_{\lis}(\bun{G}^{c_1=\alpha},\Lambda)^{\omega}
\end{equation*}
We have the following compatiblity:
\begin{lem}\label{lem: hecke action decomposition compatible}
    Under the isomorphism \(\pi_1(G)_Q\cong X^*(Z(\widehat{G})^Q)\), the Hecke operator \(T_V\) attached to a representation \(V\in \rep_{\Lambda}(\widehat{G}\rtimes W_E)_{\chi}\) sends \(A\in\D_{\lis}(\bun{G}^{c_1=\alpha},\Lambda)^{\omega}\) to \(T_V(A)\in\D_{\lis}(\bun{G}^{c_1=\alpha-\chi},\Lambda)^{\omega}\)
\end{lem}
\begin{proof}
    Consider the Hecke diagram after pullback to \(\geompt\)
    \begin{equation*}
        \begin{tikzcd}
            & {\heck_{G,C}} \arrow[ld, "\lheck"'] \arrow[rd, "\rheck"] &             \\
{\bun{G,C}} &                                                          & {\bun{G,C}}
\end{tikzcd}
    \end{equation*}
    Let \(b\in B(G)\cong |\bun{G}|\) be a point and \(\mathcal{E}_b\) the attached \(G\)-bundle on \(X_{C}\).
    Fixing a modification \(\mathcal{E}_0|_{X_C\setminus\{\infty\}}\cong \mathcal{E}_b|_{X_C\setminus\{\infty\}}\) we see that we have a pullback diagram 
    \begin{equation*}
        \begin{tikzcd}
            {\gr_{G,C}} \arrow[d, "p"'] \arrow[r, "\mathrm{BL}_{b}"] & {\heck_{G,C}} \arrow[d,"\rheck"] \\
            \{b\} \arrow[r]                                          & {\bun{G,C}}            
            \end{tikzcd}
    \end{equation*}
    By \cite[Proposition III.3.6.]{geometrization} the map \(|\gr_{G,C}|\xto{\mathrm{BL_b}}|\bun{G}|\xto{\kappa}\pi_1(G)_Q\) is induced by the map opposite to \(\pi_1(G)\to\pi_1(G)_Q\) shifted by \(\kappa(b)\) and the decomposition of \(\gr_{G,C}\) into open and closed subspaces indexed by \(\pi_1(G)\)
    Let us identify \(b\in B(G)\) with the corresponding map \(\geompt\to\bun{G,C}\).
    By \cref{cor:vanishing on components by vanishing on pullback} it suffices to check \(b^*T_V(A)=0\) for \(\kappa(b)\neq\alpha-\chi\).
    By base change we have \(b^*T_V(A)=p_{\natural}\mathrm{BL}_{b}^*(\lheck^*V\lsotimes q^*\mathcal{S}_{V})\).
    The maps \(\lheck,\rheck,q\) give rise to a map \(s\from|\heck_{G,C}|\to\pi_1(G)_Q\times\pi_1(G)_Q\times\Gamma\backslash\pi_1(G)\) recording in which connected component the left bundle, right bundle lies and what type of modification some point of the Hecke stack has.
    If \(x\) is some point in the support of \(\lheck^*V\lsotimes q^*\mathcal{S}_{V}\), then \(s(x)=(\alpha,\beta,\chi')\) for some \(\chi'\) lifting \(\chi\) by \cref{lem: geometric satake decomposition compatible}.
    It follows that \(\beta=\alpha-\chi\).
    Thus we kill \(\lheck^*V\lsotimes q^*\mathcal{S}_{V}\) by applying \(\mathrm{BL}_{b}^*\) whenever \(b\neq \beta=\alpha-\chi\), like we wanted to show.
\end{proof}
\begin{lem}\label{lem: grading indperf compatible}
    Fix a choice of Whittaker datum giving rise to a sheaf \(\mathcal{W}\in \D_{\lis}(\bun{G},\Lambda)\).
    Assume that \(\ell\) does not divide \(|\pi_0(Z(G))|\).
    Then the functor 
    \begin{equation*}
        \mathrm{Ind}\perf^{\qc}([Z^1(W_E,\widehat{G})/\widehat{G}])\to\D_{\lis}(\bun{G},\Lambda)
    \end{equation*}
    given by \(A\mapsto A *\mathcal{W}\) preserves the decomposition into \(X^*(Z(\widehat{G})^Q)\cong\pi_1(G)_Q\)-graded pieces after twisting the isomorphism \(X^*(Z(\widehat{G})^Q)\cong\pi_1(G)_Q\) by \(-1\).
\end{lem}
\begin{proof}
    Let \(P\) be an open pro-\(p\) subgroup of the wild inertia and \(W\subset W_E/P\) a discretization of the tame inertia.
    By construction of the spectral action it suffices to check the claim for \(\perf([Z^1(W/\widehat{G})/\widehat{G}])\).
    By \cite[Theorem VIII.5.1.]{geometrization} the \(\infty\)-category \(\mathrm{Ind}\perf([Z^1(W,\widehat{G})/\widehat{G}])\) is generated by \(\perf(B\widehat{G})\) under cones and retracts.
    Note that \(\perf(B\widehat{G})=\bigoplus_{\chi\in X^*(Z(\widehat{G})^Q)}\perf(B\widehat{G})_{\chi}\) where \(\perf(B\widehat{G})_{\chi}\) are those perfect complexes where \(Z(\widehat{G})^Q\) acts by \(\chi\).
    It is clear that \(f^*\from\perf(B\widehat{G})\to\mathrm{Ind}\perf([Z^1(W,\widehat{G})/\widehat{G}])\) is compatible with the \(X^*(Z(\widehat{G})^Q)\)-grading on both sides.
    It is thus sufficient to see that the functor \(\perf(B\widehat{G})\to\D_{\lis}(\bun{G},\Lambda)\) given by \(V\mapsto f^*V*\mathcal{W}\) preserves the grading as claimed.
    Note that any \(\widehat{G}\)-representation \(V\) where \(Z(\widehat{G})^Q\) acts by some character \(\chi\) embeds into \(\ind_{\widehat{G}}^{\widehat{G}\rtimes Q}V\) as a direct summand.
    Observe that \(Z(\widehat{G})^Q\) still acts via \(\chi\) on \(\ind_{\widehat{G}}^{\widehat{G}\rtimes Q}V\).
    Then by the construction of the spectral action this lemma reduces to \cref{lem: hecke action decomposition compatible}.
\end{proof}
	\section{On the categorical form of Fargues' conjecture}
\subsection{When algebra actions induce equivalences}\label{equivalence criterion}
\begin{lem}\label{lem: equivalence on endomorphism ring of tensor unit implies equivalence on perfect complexes}
    Let \(R\) and \(S\) be \(E_1\)-ring spectra.
    Let \(a\from\perf(R)\to\End(\perf(S))\) be an exact functor sending \(R\) to \(\id_{\perf(S)}\).
    Evaluation on \(S\) gives an induced functor \(f\from\perf(R)\to\perf(S)\) defined by \(f(M)=a(M)(S)\).
    If the induced map \(\End_{R}(R)\to\End_S(S)\) is an equivalence, then \(f\) is an equivalence.
\end{lem}
\begin{proof}
    Taking ind-completions we get a functor \(\LMod_R\to \mathrm{Ind}(\End(\perf(S)))\).
    Observe that in general if we have an exact functor of stable \(\infty\)-categories \(f\from\mathcal{C}\to \mathcal{D}\) and \(\mathcal{D}'\subset \mathcal{D}\) is some full stable subcategory closed under retracts, then so is \(f^{-1}(\mathcal{D}')\subset \mathcal{C}\).
    Therefore since \(\perf(R)\) is the smallest stable subcategory of \(\LMod_R\) containing \(R\) that is closed under retracts, the essential image of \(\perf(R)\) in \(\mathrm{Ind}(\End(\perf(S)))\) is contained in the smallest stable subcategory containing \(\id_{\perf(S)}\) that is closed under retracts.
    Note that \(\End^{\mathrm{ex}}(\perf(S))\) consisting of exact endofunctors is a stable subcategory of \(\mathrm{Ind}(\End(\perf(S)))\) that is closed under retracts, hence the action is actually given by a functor
    \begin{equation*}
        \perf(R)\to\End^{\mathrm{ex}}(\perf(S)).
    \end{equation*}
    By \cite[Proposition 5.5.7.8.]{htt} taking ind-categories gives us a natural functor 
    \begin{equation*}
        \End^{\mathrm{ex}}(\perf(S))\to\LEnd(\LMod_S)
    \end{equation*}
    from exact endofunctors of \(\perf(S)\) to colimit preserving endofunctors of \(\LMod_S\) that is exact. 
    For exactness, the functor is given by restricting the image of the composite 
    \begin{equation*}
        \End^{\mathrm{ex}}(\perf(S))\to\fun(\perf(S),\LMod_S)\to\fun(\LMod_S,\LMod_S)
    \end{equation*}
    The second arrow is exact, being an adjoint between stable categories, the first arrow is exact as the inclusion \(\perf(S)\to\LMod_S\) is.
    We obtain an exact monoidal functor 
    \begin{equation*}
        \perf(R)\to\LEnd(\LMod_S)
    \end{equation*}
    This extends to a colimit preserving functor 
    \begin{equation*}
        \LMod_R\to\LEnd(\LMod_S)
    \end{equation*}
    One recovers the original functor \(a\) via the equivalence \(\LEnd(\LMod_S)\simeq \End^{\mathrm{ex}}(\perf(S))\) that exists by \cite[Proposition 5.5.7.8.]{htt} and by restricting \(\LMod_R\) to \(\perf(R)\).
    Evaluation on \(S\in\LMod_S\) gives a functor \(g\from \LMod_R\to\LMod_S\).
    This is colimit preserving as colimits in functor categories are computed pointwise.
    Thus by \cite[Proposition 7.1.2.4.]{ha} there is an \(S\)-\(R\)-bimodule \(M\) such that \(g\simeq M\otimes_R -\).
    From this description it is clear that \(M\cong g(R)\) as a left \(S\)-module with right \(R\)-module structure induced by \(R^{\op}\cong\End_R(R)\to\End_S(g(R))\).
    This fully determines the structure as an \(S\)-\(R\)-bimodule by \cite[Remark 4.3.3.8.]{ha}.
    Now by assumption \(g(R)\cong S\) and the morphism \(R\cong\End_R(R)\to\End_S(g(R))\) is an equivalence.
    This means \(R\cong S\) and \(g(R)\) is equivalent to \(S\) considered as an \(S\)-\(S\)-bimodule in the natural way and then as an \(S\)-\(R\)-bimodule via the isomorphism \(R\cong S\).
    Thus \(g\) is an equivalence.
    Restricting to perfect modules implies that \(f\) is also an equivalence.
\end{proof}
\begin{cor}\label{cor: isomorphism on center equivalence of categories}
    Let \(R\) and \(S\) be commutative rings (concentrated in degree 0).
    Assume we have an exact action of \(\perf(R)\) on \(\perf(S)\) denoted by \(-*-\).
    Then we get an induced map \(R\to\pi_{0}\bcenter(\perf(S))\) by the action of \(R\) on \(R*-\cong\id_{\perf(S)}\).
    If this map is an isomorphism, then the functor \(-*S\) is an equivalence.
\end{cor}
\begin{proof}
    The following diagram commutes:
    \begin{equation*}
        \begin{tikzcd}
            \End_R(R) \arrow[r, "-*S"] & \End_S(S) \arrow[d, "a"] \\
            R \arrow[r] \arrow[u, "r"] & \pi_0\bcenter(\perf(S)) 
            \end{tikzcd}
    \end{equation*}
    Here \(r\) is the canonical map \(R\cong R^{\op}\cong\End_R(R)\) and \(a\) is induced by the action of \(\End_S(S)\) on \(S\otimes-\cong\id_{\perf(S)}\).
    By assumption the bottom arrow is and isorphism, so is \(r\) and \(s\), hence \(-*S\) satisfies the assumption of \cref{lem: equivalence on endomorphism ring of tensor unit implies equivalence on perfect complexes}.
\end{proof}
\subsection{Reduction to degree 0}\label{reduction to degree 0}
\begin{defn}
    Let \(\mathcal{W}\) denote the lisse sheaf on \(\bun{T}\) associated to the representation \(\cind_{*}^{T(E)}1\) that is attached to the unique Whittaker datum for \(T\).
\end{defn}
\begin{defn}
    Given a character \(\chi\in X^*(\widehat{T})\) let \(\oo[\chi]\) denote the line bundle on \([Z^1(W_E,\widehat{T})/\widehat{T}]\) given by \(f^*V_{\chi}\) where \(f\from [Z^1(W_E,\widehat{T})/\widehat{T}]\to B\widehat{T}\) is the natural map and \(V_{\chi}\) is the \(\widehat{T}\)-representation attached to \(\chi\) considered as a line bundle on \(B\widehat{T}\).
\end{defn}
\begin{lem}
    We have a decomposition 
    \begin{equation*}
        \perf([Z^1(W_E,\widehat{T})/\widehat{T}])\simeq\bigoplus_{[\chi]\in X^*(\widehat{T}^Q)}\perf([Z^1(W_E,\widehat{T})/\widehat{T}])_{[\chi]}
    \end{equation*}
    Additionally each \(\perf([Z^1(W_E,\widehat{T})/\widehat{T}])_{[\chi]}\) is equivalent to \(\perf([Z^1(W_E,\widehat{T})/\widehat{T}])_0\) via \(\oo[\chi]\otimes-\) for some \(\chi\in X^*(\widehat{T})\) lifting \(\chi\) and additionally \(\perf([Z^1(W_E,\widehat{T})/\widehat{T}])_0\simeq\perf(\Hom(T(E),\gm))\).
\end{lem}
\begin{proof}
    A tedious computation shows that \(\oo[\chi]\in\perf([Z^1(W_E,\widehat{T})/\widehat{T}])_{[\chi]}\).
    Then everything follows from \cref{lem: perfect complexes on gerbes}.
\end{proof}
\begin{lem}
    The functor \(-*\mathcal{W}\from\perf([Z^1(W_E,\widehat{T})/\widehat{T}])\to\D_{\lis}(\bun{T},\zl)\) restricts to a functor 
    \begin{equation*}
        \perf([Z^1(W_E,\widehat{T})/\widehat{T}])_0%
        \to\D(T(E)\mhyphen\Mod_{\zl})
    \end{equation*}
\end{lem}
\begin{proof}
    This is immediate from \cref{lem: grading indperf compatible}
\end{proof}
\begin{lem}\label{lem: reduction to degree 0 parts}
    Assume that the functor 
    \begin{equation*}
        \perf([Z^1(W_E,\widehat{T})/\widehat{T}])_0\to\D(T(E)\mhyphen\Mod_{\zl})
    \end{equation*}
    restricts to an equivalence 
    \begin{equation*}
        \perf^{\qc}([Z^1(W_E,\widehat{T})/\widehat{T}])_0\xto{\simeq}\D(T(E)\mhyphen\Mod_{\zl})^{\omega}
    \end{equation*}
    Then the functor \(-*\mathcal{W}\) restricts to an equivalence
    \begin{equation*}
        \perf^{\qc}([Z^1(W_E,\widehat{T})/\widehat{T}])\xto{\simeq}\D_{\lis}(\bun{T},\zl)^{\omega}
    \end{equation*}
\end{lem}
\begin{proof}
    Recall that \(\oo[\chi]\otimes-\) is an autoequivalence of \(\perf^{\qc}([Z^1(W_E,\widehat{T})/\widehat{T}])\) and \(\oo[\chi]*-\) is an autoequivalence of \(\D_{\lis}(\bun{T},\zl)^\omega\).
    For fully faithfulness, it suffices to show that for \(A\in\perf^{\qc}([Z^1(W_E,\widehat{T})/\widehat{T}])_{[\chi]}\) and \(B\in\perf^{\qc}([Z^1(W_E,\widehat{T})/\widehat{T}])_{[\chi']}\) we have \(\rhom(A,B)\cong\rhom(A*\mathcal{W},B*\mathcal{W})\).
    By \cref{lem: grading indperf compatible}, both sides vanish if \([\chi]\neq[\chi']\) and if \([\chi]=[\chi']\), we may even assume \([\chi]=0\) by applying \(\oo[-\chi]\otimes-\).
    Then the desired isomorphism between hom-anima holds by assumption.
    For essential surjectivity, it suffices to check that lisse sheaves \(\mathcal{F}\in\D_{\lis}(\bun{T},\zl)^{\omega}_{[\chi]}\) concentrated in a single component lie in the essential image.
    By \cref{lem: grading indperf compatible}, \(\oo[-\chi]*\mathcal{F}\) lives in the component corresponding to \(0\in B(T)\), so by assumption we have some \(A\in\perf^{\qc}([Z^1(W_E,\widehat{T})/\widehat{T}])_0\) such that \(A*\mathcal{W}\cong \oo[-\chi]*\mathcal{F}\).
    Then \((\oo[\chi]\otimes A)*\mathcal{W}\cong \mathcal{F}\), so \(\mathcal{F}\) lies in the essential image.
\end{proof}
\subsection{Reduction to commutative algebra}\label{matching of union}
\begin{lem}\label{lem: T(E)/K-mod is full subcategory of T(E)-mod}
    Let \(K\) be an open pro-\(p\) subgroup of \(T(E)\).
    Then the functor 
    \begin{equation*}
        T(E)/K\mhyphen\Mod_{\zl}\to T(E)\mhyphen\Mod_{\zl}
    \end{equation*}
    induced by \(T(E)\to T(E)/K\) induces a fully faithful functor 
    \begin{equation*}
        \D(T(E)/K\mhyphen\Mod_{\zl})\to \D(T(E)\mhyphen\Mod_{\zl})
    \end{equation*}
\end{lem}
\begin{proof}
    Well-known.
\end{proof}
\begin{lem}\label{lem: compact T(E)-representations as filtered union of perfect modules}
    Let \(K_i\), \(i\in I\) form a neighborhood basis by open pro-\(p\) subgroups of the identity in \(T(E)\).
    Identifying \(\D(T(E)/K_i\mhyphen\Mod_{\zl})\) as a full subcategory of \(\D(T(E)\mhyphen\Mod_{\zl})\) via \cref{lem: T(E)/K-mod is full subcategory of T(E)-mod}
    we obtain 
    \begin{equation*}
        \bigcup_{i\in I}\D(T(E)/K_i\mhyphen\Mod_{\zl})^\omega=\D(T(E)\mhyphen\Mod_{\zl})^\omega
    \end{equation*}
\end{lem}
\begin{proof}
    The inclusion \(T(E)/K_i\mhyphen\Mod_{\zl}\subset T(E)\mhyphen\Mod_{\zl}\) has both adjoints, so the inclusion 
    \begin{equation*}
        \D(T(E)/K_i\mhyphen\Mod_{\zl})\subset\D(T(E)\mhyphen\Mod_{\zl})
    \end{equation*}
    has both adjoints too.
    It follows that this inclusion commutes with small limits and colimits.
    A similar story holds for \(\D(T(E)/K_i\mhyphen\Mod_{\zl})\subset\D(T(E)/K_j\mhyphen\Mod_{\zl})\) whenever \(K_j\subseteq K_i\).
    It follows that the union \(\bigcup_{i\in I}\D(T(E)/K_i\mhyphen\Mod_{\zl})^\omega\) is the thick stable subcategory generated by compact inductions \(\cind_{K_i}^{T(E)}\zl\).
    This also describes \(\D(T(E)\mhyphen\Mod_{\zl})^\omega\).
\end{proof}
\begin{lem}\label{lem: idempotents match up}
    The functor 
    \begin{equation*}
        -*\mathcal{W}\from\perf([Z^1(W_E,\widehat{T})/\widehat{T}])\to\D_{\lis}(\bun{T},\zl)=\prod_{b\in B(T)}\D(T(E)\mhyphen\Mod_{\zl})
    \end{equation*}
    given by acting on the Whittaker sheaf \(\mathcal{W}\) restricts to 
    \begin{equation*}
        \perf([Z^1(W_{F/E}/P,\widehat{T})/\widehat{T}])\to\left(\prod_{b\in B(T)}\D(T(E)/\Theta(P)\mhyphen\Mod_{\zl})\right)^\omega
    \end{equation*}
    where \(P\) is some open subgroup of \(F^\times\) such that \(\Theta(P)\) is pro-\(p\).
    It is given by letting \(\perf([Z^1(W_{F/E}/P,\widehat{T})/\widehat{T}])\) act on \(\zl[T(E)/\Theta(P)]\).
\end{lem}
\begin{proof}
    Let \(e_P\in\oo(Z^1(W_{F/E},\widehat{T}))^{\widehat{T}}\) be the idempotent corresponding to the open and closed substack \([Z^1(W_{F/E}/P,\widehat{T})/\widehat{T}]\subset [Z^1(W_{F/E},\widehat{T})/\widehat{T}]\).
    Then for any perfect complex \(\mathcal{F}\in \perf([Z^1(W_{F/E}/P,\widehat{T})/\widehat{T}])\) we have \(\mathcal{F}= e_P\oo\otimes \mathcal{F}\), where \(\oo\) is the structure sheaf of \([Z^1(W_E,\widehat{T})/\widehat{T}]\).
    Thus \(-*\mathcal{W}\) restricts to a functor 
    \begin{equation*}
        \perf([Z^1(W_{F/E}/P,\widehat{T})/\widehat{T}])\to \mathcal{C}
    \end{equation*}
    where \(\mathcal{C}\) is the full subcategory of \(\D_{\lis}(\bun{T},\zl)\) spanned by those lisse sheaves \(A\) such that \(e_P\oo*A=A\).
    Additionally this functor is given by acting on \(e_P\oo*\mathcal{W}\).
    Thus our goal is to identify \(\mathcal{C}\) and \(e_P\oo*\mathcal{W}\).
    
    Under the isomorphism \(\oo(Z^1(W_{F/E},\widehat{T}))^{\widehat{T}}\cong\oo(\Hom(T(E),\gm))\)
    the idempotent \(e_P\) cuts out the open and closed subscheme \(\Hom(T(E)/\Theta(P),\gm)\subset\Hom(T(E),\gm)\) by the proof of \cref{lem: geometric incarnation langlands duality tori}.
    Applying \cite[Proposition IX.6.5.]{geometrization}, \cref{rem: our iso Z^spec(T) is the Fargues-Scholze one} and \cref{lem: excursion algebra action compatible with spectral action}, we have \(e_P\oo*A=A\) for a lisse sheaf \(A\) corresponding to \((A_b)_{b\in B(T)}\) for \(A_b\in\D(T(E)\mhyphen\Mod_{\zl})\) if and only if \(A_b\xto{e_P}A_b\) is the identity for all \(b\in B(T)\).
    By \cref{cor: subcategory given by e_P}, we have
    \begin{equation*}
        \mathcal{C}=\prod_{b\in B(T)}\D(T(E)/\Theta(P)\mhyphen\Mod_{\zl})
    \end{equation*}
    One easily sees that \(e_P\oo\) is a \(\otimes\)-idempotent in the symmetric monoidal \(\infty\)-category \(\perf([Z^1(W_E,\widehat{T})/\widehat{T}])\).
    Therefore \(e_P\oo*-\) is the left adjoint to \(\mathcal{C}\injto\D_{\lis}(\bun{T},\zl)\).
    This is explicitly given by taking coinvariants for \(\Theta(P)\).
    \cref{rem: whittaker sheaf for tori} tells us that \(e_P\oo*\mathcal{W}=\zl[T(E)/\Theta(P)]\).
    Finally the lemma follows from the fact that \(\zl[T(E)/\Theta(P)]\) is compact and the spectral action preserves compact objects.
\end{proof}
\begin{cor}\label{cor: reduction of categorical Fargues's conjecture to commutative algebra}
    The functor 
    \begin{equation*}
        \perf([Z^1(W_{F/E}/P,\widehat{T})/\widehat{T}])\to\left(\prod_{b\in B(T)}\D(T(E)/\Theta(P)\mhyphen\Mod_{\zl})\right)^\omega
    \end{equation*}
    restricts to a functor 
    \begin{equation*}
        \perf([Z^1(W_E/P,\widehat{T})/\widehat{T}])_0\to\D(T(E)/\Theta(P)\mhyphen\Mod_{\zl})^\omega
    \end{equation*}
    where \(P\) is any open subgroup of \(F^\times\) such that \(\Theta(P)\) is pro-\(p\).
\end{cor}
\subsection{Proof for tori}
We now have all the ingredients to prove the categorical form of Fargues's conjecture for tori.
\begin{thm}\label{thm: categorical form Fargues' conjecture tori}
	The categorical form of Fargues's conjecture holds in the following form:
	For any prime \(\ell\) there is an equivalence of stable $\infty$-categories
	\[
		\D^{b,\mathrm{qc}}_{\coh,\nilp}([Z^1(W_E,\widehat{T})/\widehat{T}])\simeq \D_{\lis}(\bun{T},\zl)^\omega
	\]
	equipped with their actions of $\perf([Z^1(W_E,\widehat{T})/\widehat{T}])$. After taking ind-categories or restricting to connected components the structure sheaf gets mapped to the (unique choice of) Whittaker sheaf \(\mathcal{W}\).
\end{thm}
\begin{proof}
    Letting \(\perf([Z^1(W_E,\widehat{T})/\widehat{T}])\) act on \(\mathcal{W}\) gives a functor 
    \begin{equation*}
        -*\mathcal{W}\from \perf([Z^1(W_E,\widehat{T})/\widehat{T}])\to\D_{\lis}(\bun{T},\zl)
    \end{equation*}
    Recall that \(\D^{b,\mathrm{qc}}_{\coh,\nilp}([Z^1(W_E,\widehat{T})/\widehat{T}])=\perf^{\qc}([Z^1(W_E,\widehat{T})/\widehat{T}])\), as the nilpotent cone in \(\widehat{\mathfrak{t}}^*\) is \(\{0\}\) and \cite[Theorem VIII.2.9]{geometrization}.
    We claim that the restriction 
    \begin{equation*}
        -*\mathcal{W}\from \perf^{\qc}([Z^1(W_E,\widehat{T})/\widehat{T}])\to\D_{\lis}(\bun{T},\zl)
    \end{equation*}
    is fully faithful with essential image the compact objects in \(\D_\lis(\bun{T},\zl)\).
    Then the compatibility assumption holds automatically.
    Thus to prove the theorem we need to check this claim.
    \cref{lem: reduction to degree 0 parts} tells us that we only need to check that the restriction 
    \begin{equation*}
        -*\mathcal{W}\from \perf^{\qc}([Z^1(W_E,\widehat{T})/\widehat{T}])_0\to\D(T(E)\mhyphen\Mod_{\zl})^{\omega}
    \end{equation*}
    is an equivalence.
    We know that
    \begin{equation*}
        (\perf^{\qc}([Z^1(W_E,\widehat{T})/\widehat{T}])=\bigcup_{P\subset F^\times}\perf^{\qc}([Z^1(W_{F/E}/P,\widehat{T})/\widehat{T}])
    \end{equation*}
    where \(P\) runs over open subgroups of \(F^\times\).
    By \cref{lem: condensed H_1 of quotient is torus quotient} we may additionally assume that the \(\Theta(P)\) are pro-\(p\).
    Then using \cref{lem: compact T(E)-representations as filtered union of perfect modules} and \cref{cor: reduction of categorical Fargues's conjecture to commutative algebra} we only need to check that the restriction 
    \begin{equation*}
        \perf([Z^1(W_E/P,\widehat{T})/\widehat{T}])_0\to\D(T(E)/\Theta(P)\mhyphen\Mod_{\zl})^\omega
    \end{equation*}
    is an equivalence, for \(P\) as above.
    Having the spectral action we get a ring map 
    \begin{equation*}
        \oo(Z^1(W_E,\widehat{T}))^{\widehat{T}}\to\pi_0\bcenter(\D(T(E)\mhyphen\Mod_{\zl}))
    \end{equation*}
    induced by the action of \(\oo(Z^1(W_E,\widehat{T})\) on \(\oo*-\) which under the identification
    \begin{gather*}
        \oo(Z^1(W_E,\widehat{T}))^{\widehat{T}}=\oo(\Hom(T(E),\gm))=\varprojlim_{P\subset F^\times}\zl[T(E)/\Theta(P)]
        \shortintertext{and}\\
        \pi_0\bcenter(\D(T(E)\mhyphen\Mod_{\zl}))=\varprojlim_{P\subset F^\times}\zl[T(E)/\Theta(P)]
    \end{gather*}
    is an isomorphism.
    This follows from \cite[Proposition IX.6.5.]{geometrization}, \cref{rem: our iso Z^spec(T) is the Fargues-Scholze one} and \cref{lem: excursion algebra action compatible with spectral action} (note that while \cite[Proposition IX.6.5.]{geometrization} talks about the spectral Bernstein center, it actually computes the action of the excursion algebra).
    Under this isomorphism both \(\perf([Z^1(W_E/P,\widehat{T})/\widehat{T}])\) and \(\prod_{b\in B(T)}\D(T(E)/\Theta(P)\mhyphen\Mod_{\zl})\) are given by those objects \(A\) in \(\perf([Z^1(W_E/P,\widehat{T})/\widehat{T}])\) or \(\D(T(E)\mhyphen\Mod_{\zl})\) such that \(e_P\) acts by the identity for some idempotent \(e_P\in\varprojlim_{P\subset F^\times}\zl[T(E)/\Theta(P)]\).
    The action of \(\oo(Z^1(W_E/P,\widehat{T})^{\widehat{T}}\) on \(\oo*-\) induces a ring map \(\oo(Z^1(W_E/P,\widehat{T}))^{\widehat{T}}\to\pi_0\bcenter(\D(T(E)/\Theta(P)\mhyphen\Mod_{\zl}))\), which fits into the following commutative diagram
    \begin{equation*}
        \begin{tikzcd}
            {\oo(Z^1(W_E,\widehat{T})^{\widehat{T}}} \arrow[r] \arrow[d, "e_P"'] & \pi_0\bcenter(\D(T(E)/\mhyphen\Mod_{\zl})) \arrow[d, "e_P"] \\
            {\oo(Z^1(W_E/P,\widehat{T})^{\widehat{T}}} \arrow[r]                       & \pi_0\bcenter(\D(T(E)/\Theta(P)\mhyphen\Mod_{\zl}))             
            \end{tikzcd}
    \end{equation*}
    As observed before the top horizontal arrow is an isomorphism and as \(e_P\) is an idempotent it follows that the bottom horizontal arrow is an isomorphism too.
    Note that 
    \begin{equation*}
        \perf([Z^1(W_E/P,\widehat{T})/\widehat{T}])_0=\perf(\zl[T(E)/\Theta(P)])=\D(T(E)/\Theta(P)\mhyphen\Mod_{\zl})^{\omega}
    \end{equation*}
    and the functor 
    \begin{equation*}
        -*\mathcal{W}\from\perf([Z^1(W_E/P,\widehat{T})/\widehat{T}])_0\to\D(T(E)/\Theta(P)\mhyphen\Mod_{\zl})^\omega
    \end{equation*}
    is given by letting \(\perf([Z^1(W_E/P,\widehat{T})/\widehat{T}])_0\) act on \(\zl[T(E)/\Theta(P)]\).
    Thus we may apply \cref{cor: isomorphism on center equivalence of categories} and finish the proof.
\end{proof}
Let us finally discuss \(t\)-exactness of the equivalence.
For this we will need to understand the spectral action, so let us start with analyzing the Hecke stack.
We recall some of its properties, for details we refer to the discussion after \cite[Remark 3.2.]{geometric-Eisenstein}.
\begin{lem}\label{lem: explicit hecke stack}
    We have an isomorphism 
    \begin{equation*}
        \heck_{T}\cong\bun{T}\times\coprod_{\mu\in X_*(T)/Q}\Div^1_{E_{\mu}}
    \end{equation*}
    where \(E_\mu\) is the reflex field of \(\mu\).
    Under this isomorphism \(\lheck\) identifies with the projection to the first factor.
    Let us write \(\rheck=\coprod_{\mu\in X_*(T)/Q}\rheck_\mu\) with \(\rheck_\mu\from\bun{T}\times\Div^1_{E_\mu}\to\bun{T}\times\Div^1\).
    Then \(\rheck_\mu=s_\mu\times\mathrm{can}\) where \(\mathrm{can}\from\Div^1_{E_\mu}\to\Div^1_E\) is the canonical map and \(s_\mu\) is the composite \(\bun{T}\xto{\cong}\bun{T_{b_{-\mu'}}}\cong\bun{T}\).
    Here \([\mu]\) is the image of \(\mu\) in \(X_*(T)_Q\), the map \(\bun{T}\xto{\cong}\bun{T_{b_{-[\mu]}}}\) is the isomorphism in \cite[Corollary III.4.3.]{geometrization} and \(\bun{T_{b_{-[\mu]}}}\cong\bun{T}\) is induced by \(T_{b_{-[\mu]}}\cong T\).
\end{lem}
\begin{lem}\label{lem: explicit satake sheaf}
    For a character \(\chi\in X^*(\widehat{T})\), the sheaf \(q^*\mathcal{S}_\chi\) on \(\heck_G\) via geometric Satake as in the discussion after \cref{defn: hecke action} is the constant sheaf on the component \(\overline{\chi}\), where \(\overline{\chi}\) is the image of \(\chi\) under the map \(X^*(\widehat{T})\cong X_*(T)\to X_*(T)/Q\).
\end{lem}
\begin{cor}
    The action of \(\oo[\chi]\) is \(t\)-exact for the natural \(t\)-structure on \(\D_{\lis}(\bun{T},\zl)\).
\end{cor}
\begin{proof}
    Using \cref{lem: explicit hecke stack} and \cref{lem: explicit satake sheaf} we obtain that \(\oo[\chi]\) identifies with \(\ind_{W_{E_\chi}}^{W_E}\circ s_{\chi*}\) in the notation of \cref{lem: explicit hecke stack}, which is \(t\)-exact as \(s_{\chi*}\) is and \(\ind_{W_{E_\chi}}^{W_E}\) is \(t\)-exact by \cref{lem: ambidexterity condensed group representation} (using that \(W_{E}/W_{E_{\chi}}\) is a discrete finite set).
\end{proof}
\begin{thm}\label{thm: t-exact}
    The equivalence \cref{thm: categorical form Fargues' conjecture tori} is \(t\)-exact for the natural \(t\)-structures on both sides.
\end{thm}
\begin{proof}
    Recall that the overall strategy to prove \cref{thm: categorical form Fargues' conjecture tori} was to use \cref{lem: reduction to degree 0 parts} and some discussion on the idempotents of \(\oo(Z^1(W_E,\widehat{T}))^{\widehat{T}}\) as in \cref{lem: idempotents match up} to reduce to the question whether
    \begin{equation*}
        \perf([Z^1(W_E/P,\widehat{T})/\widehat{T}])_0\to\D(T(E)/\Theta(P)\mhyphen\Mod_{\zl})^\omega
    \end{equation*}
    is an equivalence.
    We observe that \(\mathcal{W}^{\Theta(P)}\) is compact projective, so \(\Hom(\mathcal{W}^{\Theta(P)},-)\) gives rise to a \(t\)-exact equivalence \(\D(T(E)/\Theta(P)\mhyphen\Mod_{\zl})^\omega\simeq\perf(\End(\mathcal{W}^{\Theta(P)}))\).
    Note that \(\End(\mathcal{W}^{\Theta(P)})\cong\zl[T(E)/\Theta(P)]\) and we also have \(\perf([Z^1(W_E/P,\widehat{T})/\widehat{T}])_0\simeq\perf(\oo(Z^1(W_E/P,\widehat{T}))^{\widehat{T}})\) and that under these identifications the composite 
    \begin{equation*}
        \perf(\oo(Z^1(W_E/P,\widehat{T}))^{\widehat{T}})\simeq\perf([Z^1(W_E/P,\widehat{T})/\widehat{T}])_0\to\D(T(E)/\Theta(P)\mhyphen\Mod_{\zl})^\omega\simeq\perf(\zl[T(E)/\Theta(P)])
    \end{equation*}
    is induced by the isomorphism of rings \(\oo(Z^1(W_E/P,\widehat{T}))^{\widehat{T}}\cong\zl[T(E)/\Theta(P)]\), in particular it is \(t\)-exact.

    We finish the proof by observing that since \(\oo[\chi]\) are \(t\)-exact, we can strengthen the claim of \cref{lem: reduction to degree 0 parts} to say that \(-*\mathcal{W}\) is a \(t\)-exact equivalence of categories of its restriction to \(\perf([Z^1(W_E/P,\widehat{T})/\widehat{T}])_0\) is.
\end{proof}
\begin{rem}
    Using \cref{thm: t-exact}, we can make the equivalence of \cref{thm: categorical form Fargues' conjecture tori} very explict.
    Namely start with a complex \(A\in\perf^\qc([Z^1(W_E,\widehat{T})/\widehat{T}])\).
    This splits to a direct according to the the character \(\chi\in X^*(\widehat{T}^Q)\) the action of \(\widehat{T}^Q\) factors.
    The equivalence is compatible with direct sums, so let us assume that the action of \(\widehat{T}^Q\) on \(A\) factors through a single character \(\chi\).
    Additionally, there is an action of \(\oo(Z^1(W_E,\widehat{T}))^{\widehat{T}}\) on \(A\).
    Via the isomorphism \(\oo(Z^1(W_E,\widehat{T}))^{\widehat{T}}=\oo(\Hom(T(E),\gm))=\varprojlim_{P\subset F^\times}\zl[T(E)/\Theta(P)]\) this gives rise to a complex \(A\in\D(T(E)\mhyphen\Mod_{\zl})\).
    We consider this a sheaf on \(\bun{T}\) by putting it on the stratum \(\bun{T}^{b_{-\chi}}\), where \(b_{-\chi}\) is the (basic) element of the Kottwitz set attached to \(-\chi\) via the isomophism \(X^*(\widehat{T})_Q\cong\pi_1(T)_Q\cong B(T)\).
\end{rem}
\begin{rem}
    Let \(\Lambda\) be a \(\zl\)-algebra.
    By applying \(-\otimes_{\perf(\zl)}\perf(\Lambda)\) to \cref{thm: categorical form Fargues' conjecture tori}, the theorem extends to any \(\zl\)-algebra.
    One can then execute the proof of \cref{thm: t-exact} with \(\Lambda\) instead of \(\zl\), using that formation of \(\oo(Z^1(W_E/P,\widehat{T}))^{\widehat{T}}\) commutes with any base change.

    In conclusion both \cref{thm: categorical form Fargues' conjecture tori} and \cref{thm: t-exact} extend to any \(\zl\)-algebra \(\Lambda\) from the case of \(\Lambda=\zl\) via formal considerations.
\end{rem}
	\clearpage
	\printbibliography
\end{document}